\documentclass[11pt]{amsart}
\usepackage{color,graphicx,array, amssymb, amscd,slashed}
\usepackage[all]{xy}
\usepackage{tikz-cd}
\usepackage{pb-diagram}
\newtheorem{Theorem}{Theorem}[section]
\newtheorem{Lemma}[Theorem]{Lemma}
\newtheorem{Proposition}[Theorem]{Proposition}
\newtheorem{Corollary}[Theorem]{Corollary}
\theoremstyle{definition}
\newtheorem{Definition}{Definition}
\theoremstyle{remark}

\newtheorem{Remark}[Theorem]{Remark} 

\numberwithin{equation}{section}
\newcommand{\R}{\mathbb R}

\newcommand{\C}{\mathbb C}

\newcommand{\D}{\mathbb D}

\newcommand{\su}{\mathfrak{su}_3}

\newcommand{\g}{\mathcal {G}}

\newcommand{\U}{{\rm U}_{3}}
\newcommand{\SU}{{\rm SU}_{3}}
\newcommand{\SUt}{{\rm SU}_{2}}

\newcommand{\SO}{{\rm SO}_{3}}

\newcommand{\Uone}{{\rm U}_1}

\newcommand{\SL}{{\rm SL}_3 \mathbb C}

\renewcommand{\sl}{\mathfrak{sl}_3 \mathbb C}

\newcommand{\lsu}{\Lambda \mathfrak{su}_3}

\newcommand{\ad}{\operatorname{Ad}}
\newcommand{\di}{\operatorname{diag}}
\newcommand{\tr}{\operatorname{trace}}
\newcommand{\Fi}{\operatorname{Fix}}

\renewcommand{\Im}{\operatorname {Im}}

\newcommand{\bpm}{\begin{pmatrix}}
\newcommand{\epm}{\end{pmatrix}}

\newcommand{\CP}{\mathbb {CP}^{2}}
\newcommand{\f}{\mathfrak f}

\newcommand{\LGr}{\operatorname{L_{Gr}}}
\newcommand{\SLGr}{\operatorname{SL_{Gr}}}

\usepackage{amsmath}	
\begin{document}
\title{Ruh-Vilms Theorems For Minimal Surfaces Without Complex Points and Minimal Lagrangian  Surfaces in  $\C P^2$}
 
 \author[J. F.~Dorfmeister]{Josef F. Dorfmeister}
 \address{Fakult\"at f\"ur Mathematik, 
 TU-M\"unchen, 
 Boltzmannstr.3,
 D-85747, 
 Garching, 
 Germany}
 \email{dorfm@ma.tum.de}

 \author[S.-P.~Kobayashi]{Shimpei Kobayashi}
 \address{Department of Mathematics, Hokkaido University, 
 Sapporo, 060-0810, Japan}
 \email{shimpei@math.sci.hokudai.ac.jp}
 \thanks{The second named author is partially supported by JSPS KAKENHI Grant Number JP18K03265. }
 
\author{Hui Ma}
\address{Department of Mathematical Sciences, Tsinghua University, Beijing 100084, P.R.China} 
\email{ma-h@tsinghua.edu.cn}
\thanks{The third named author is partially supported by NSFC No.11831005, No.11671223 and No. 11961131001.}
 \date{\today}
\pagestyle{plain}
\begin{abstract}
 In this paper we investigate surfaces in $\C P^2$ without complex points and 
 characterize the minimal surfaces without complex points and the minimal 
 Lagrangian surfaces by  Ruh-Vilms type theorems. We also discuss the liftability of  an immersion from a surface to $\C P^2$ into $S^5$ in Appendix \ref{appendix}.
\end{abstract}
\maketitle

\section*{Introduction }

In recent years minimal Lagrangian surfaces in $\C P^2$ have been studied intensively (see \cite{CU, MM, McIn, Mironov, O}, etc.). It turned out that an automorphism $\sigma$ 
 of $\sl$ of order $6$ is of crucial importance. Similar investigations have used the restriction of $\sigma$ to real forms of $\sl$ and have discussed the surface classes of minimal Lagrangian surfaces in $\C \mathbb{H}^2$ \cite{McI-Loftin} and definite affine spheres \cite{Do-EWang}. Also the class of indefinite affine spheres and timelike minimal Lagrangian surfaces in the indefinite complex hyperbolic 2-space have been investigated in a similar way \cite{DE, DoKo}. Moreover, those classes of surfaces 
 have a unified picture by using real forms of 
the Kac-Moody Lie algebra of $A_2^{(2)}$ type \cite{DFKW}. While $\sigma$ arose naturally in classical geometric investigations, the question arose, whether also $\sigma^2$ and $\sigma^3$ have a simple geometric meaning.

The starting point for an approach to this question was the paper \cite{MM}, which investigated arbitrary immersions from Riemann surfaces to $\C P^2$ without complex points.
However, since immersions from $S^2$ to $\C P^2$ have been investigated intensively,  for the final goals of this paper we exclude the Riemann surface $S^ 2$ from our discussion.
More precisely, we consider an immersion $f:M \rightarrow \C P^2$ without complex points, where $M$ is a Riemann surface different from $S^2$. For our approach it is crucial to lift $f$ to a map $\f :M \rightarrow S^5$ such that $ f = \pi \circ \f$ where $\pi : S^5 \rightarrow \C P^2$ denotes the Hopf fibration. To clarify, when such a lift exists we have proven in the appendix that for a non-compact Riemann surface such a lift always exists and that in the case of a compact Riemann surface either the given immersion already has a global lift to $S^5$ or one can
 find a threefold covering 
$ \tau: \hat{M} \rightarrow M$ of $M$ such that the immersion 
$\hat{f} = f \circ \tau : \hat{M} \rightarrow \C P^2$ admits a global lift to $S^5$.

So for this paper we always assume that any immersion under consideration does have 
a global lift to $S^5$.
For a more detailed investigation of \emph{liftable} immersions $f:M \rightarrow \C P^2$ with global lift  $\f :M \rightarrow S^5$, we consider, to begin with, their composition with the universal covering
$\tilde \pi : \D \rightarrow M$ of $M$. In other word, we first investigate the case, where $M = \D$ is simply-connected. 

In this setting the ideas presented in \cite{MM} is applied. However, while in 
loc.cit. the investigation quickly moved on to consider minimal Lagrangian 
tori in $\C P^2$, in the present paper we consider a natural $\SU$-frame  
$\mathcal{F}(\f)$ and thus obtain a setting similar to the one used in \cite{DoMa1}.

In particular, two lifts with $\SU$-frame only differ by a cubic root.
Moreover, the Maurer-Cartan form of the frame $\mathcal{F}(\f)$  clearly displays the
natural invariants of an immersion into $\C P^2$ without complex points.

To understand what surface classes correspond to $\sigma, \sigma^2$ and $\sigma^3$
we apply the notion of a \emph{primitive harmonic map} relative to some automorphism of 
$\sl$. The corresponding theory, basically due to Black \cite{B},  is  collected in the first three subsections of section \ref{sc:alg}. Then we prove (Theorem \ref{equivprimitive}) that the lift 
$\f :M \rightarrow S^5$  
of a liftable immersion  $f: M \rightarrow \C P^2$ without complex points is primitive harmonic relative to $\sigma, \sigma^2$ and $\sigma^3$ respectively if and only if $f$ is minimal Lagrangian, minimal without complex points and minimal Lagrangian or flat homogeneous, respectively.
It is natural to ask, whether actually any primitive harmonic map relative to 
$\sigma$, $\sigma^2$ and $\sigma^3$ is associated to an immersion into $\C P^2$ without complex points. This is assured in Theorem \ref{thm:primitive}.

The last part of the paper answers a natural question arising from the above: when considering primitive harmonic maps one singles out special immersions among a larger class of immersions and the frames of these immersions project to $k$-symmetric spaces like a 
\lq\lq Gauss like map\rq\rq. How does this work out for the surface classes considered in this paper?

We start by considering spaces $FL_j$, $j = 1, 2, 3$ similar to \cite{McI-Loftin}. Thus we obtain three $6$-symmetric spaces of dimension $7$ which all are actually equivariantly isomorphic to $\SU/\Uone$ (Theorem \ref{Thm:3.3} and Corollary \ref{cor:3.4}). From these spaces we obtain natural projections to four different spaces (Theorem \ref{Thm:3.5}). Two of these are equivariantly diffeomorphic symmetric spaces relative to $\sigma^3$ of dimension $5$ and two are equivariantly diffeomorphic $3$-symmetric spaces relative to $\sigma^2$ of dimension $6$. 
 Let  $\SLGr(3, \C), \widetilde{\SLGr(3, \C)}$ denote 
the symmetric spaces above and $\widetilde{Fl_2}, Fl_2$ the $3$-symmetric spaces above.
 Then for any immersion $f: \D \rightarrow \C P^2$ without complex points and with lift $\f : \D \rightarrow S^5$ and 
 $\SU$-frame $\mathcal{F}(\f)$ we define a \emph{Gauss type map}  $\mathcal{G}_{j}$ to $FL_{j}$, 
 $\mathcal{H}_{1}$ to $\SLGr(3, \C)$, $\mathcal{H}_{2}$ to $Fl_{2}$, $\mathcal{H}_{3,1}$ to  $ \widetilde{\SLGr(3, \C)}$ and $\mathcal{H}_{3,2}$ to  $ \widetilde{Fl_{2}}$ given by the natural projection of $\mathcal{F}(\f)$.  
 We finally  prove the  Ruh-Vilms Theorems for $\sigma$, $\sigma^2$ and $\sigma^3$ (Theorem \ref{Thm:3.6}), characterizing minimal Lagrangian surfaces and minimal surfaces without complex points by the primitive harmonicity of the corresponding Gauss maps.
In appendix \ref{appendix}, 
 we discuss the liftability of an immersion $f: M \to \C P^2$ into $S^5$.


\section{Theory of Surfaces in $\C P^2$}

\subsection{Basic Definitions: the metrics}

Let $\C P^2$ denote the $2$-dimensional complex projective space endowed with the
Fubini-Study metric  
of constant holomorphic sectional curvature  $4$.

For the convenience of the reader we recall the definition.
For this we will use the natural  $\C$-bilinear quadratic form
\begin{equation*}
Z \cdot W = \sum_{k=1}^{3} z_kw_k,
\end{equation*}
where $Z = (z_1,z_2,z_3)$, $W = (w_1, w_2, w_3) \in \C^3$.

Then it is well known that the Fubini-Study metric can be given in \emph{homogeneous coordinates} by the formula:
\begin{equation} \label{Fubini-Study-homogeneous}
ds^2 = \frac{|Z|^2 |dZ|^2 - (\bar{Z} \cdot dZ) (Z \cdot d \bar{Z})}{|Z|^4},
\end{equation}
where $Z$ is a local holomorphic section of the tautological bundle of $\C P^2$.

Now it is an easy computation to show that equation
\eqref{Fubini-Study-homogeneous} is unchanged, if one replaces 
 $Z$ by $hZ$, where $h$ is any scalar $C^\infty$-function with values in $\C^*$.
 As a consequence, we can replace $Z$ by $Z/|Z|$ and thus obtain:
\begin{equation} \label{Fubini-Study-sphere}
ds^2 =  |dZ|^2 - (\bar{Z} \cdot dZ) (Z \cdot d \bar{Z}) = (dZ - (d Z \cdot \bar{Z})Z) \otimes
(d \bar{Z} -  (d \bar{Z} \cdot Z) \bar{Z}).
\end{equation}
Note that now $Z$ maps into the unit sphere $S^5$ in $\C^3$. Also note that we will obtain the same expression if we replace here $Z$ by $hZ$, where $h$ is any $\C^\infty$-function with values in $S^1$.

Let $\pi: S^5 \rightarrow \C P^2$ be the Hopf-projection, $p \rightarrow [p]$. Then $\pi$ is a Riemannian submersion, if one considers the metric on $S^5$ induced from the standard Hermitian product on $\C^3$ and the Fubini-Study metric on $\C P^2.$

\subsection{Liftable surfaces in $\C P^2$}\label{susc:1.2}

Let $M$ be a Riemann surface different from $S^2$ and 
$f: M \rightarrow \C P^2$  a conformal immersion. We will  write the induced metric locally in  the form
\begin{equation}
g = 2 e^\omega dz d\bar{z},
\end{equation}
 for some real valued function $\omega$.
 
 For the approach used in this paper we will need lifts of $f:M \rightarrow \C P^2$ to 
 $\f:M \rightarrow S^5$.
 
 \begin{Lemma}
 Let $M$ be a contractible Riemann surface and $f: M \rightarrow \C P^2$  a conformal immersion, then there exists a  conformal immersion $\f: M \rightarrow S^5$  such that $ f = \pi \circ \f$ holds.
 \end{Lemma}
 \begin{proof}
 We consider the pullback  $f^*S^5$ of the Hopf fibration to $M$. Then 
 $f^*S^5$ is an $S^1$-bundle over $M$. But fiber  bundles over a contractible base are trivial. Therefore there exists a section $s: M  \rightarrow f^*S^5$ and the composition of $s$ with the natural map from  $f^*S^5$ to $S^5$ yields the desired map.
 \end{proof}

 For general Riemann surfaces $M$  and general conformal immersions 
 $f: M \rightarrow \C P^2$ such a (global) lift  $\f:M \rightarrow S^5$ may not exist.

As a consequence, papers considering surfaces as we do in this paper usually 
 restrict $f$ to contractible open subsets $U$ of $M$. Since it is not clear how one can glue these maps for different $U's$ together, we will not follow this approach, but rather consider exclusively liftable immersions, where an immersion $f:M \rightarrow \C P^2$ is called \emph{liftable}  if there exists an immersion $\f: M \rightarrow S^5$  such that $f = \pi \circ \f$ holds.
 
 Note, two lifts of some immersion $f$ differ by a scalar function which takes values in $S^1$.
In the rest of this paper we will always use conformal liftable surfaces.
 
 
 \subsection{The basic invariants for surfaces $f:\D \rightarrow \C P^2$ in terms of $\f$}\label{susc:1.3}
 
 Let   $f_M :M \rightarrow \C P^2$ be a liftable immersion and $\f_M: M \rightarrow S^5$  a lift of $f$.
 
 Using the universal cover $ \tilde{\pi}: \D \rightarrow M$ of $M$ we will also consider the immersion
 $\tilde{f}_M = f_M \circ \tilde{\pi}$ and write  
 $ \tilde{\f}_M = \f :\D \rightarrow S^5$.

We consider next  the following diagram.
\[
\begin{tikzcd}[column sep=4em,row sep=4em]
\D  \ar{r}{\f}
    \ar{d}[swap]{\tilde \pi}   &   S^5   \ar{d}{\pi}   \\
M   \ar{r}{f_M} \ar{ur}{\f_M}   &   \C P^2
\end{tikzcd}
\]
\begin{Lemma}
One can choose without loss of generality $\f$ such that the diagram commutes.
\end{Lemma}

\begin{proof}
We have chosen $\f$ such that the diagram
\[
\begin{tikzcd}[column sep=4em, row sep=4em]
\D  \ar{r}{\f}
    \ar{rd}[swap]{\tilde f_M}   &   S^5   \ar{d}{\pi}   \\
   &   \C P^2
\end{tikzcd}
\]
commutes, i.e., $\pi \circ \f = \tilde{f}_M$.  We also have the relations  $ f_M = \pi \circ \f_M$ 
and  $\tilde{f}_M = f_M \circ \tilde{\pi}$ and write $ \tilde{\f}_M = \f :\D \rightarrow S^5$.
 
 It suffices to prove that the diagram  
\[
\begin{tikzcd}[column sep=4em,row sep=4em]
\D  \ar{r}{\f}\ar{d}[swap]{\tilde \pi}   &   S^5   \\
 M   \ar{ur}[swap]{\f_M}    &  
\end{tikzcd}
\]
commutes. We observe $\pi \circ  \f_M  \circ \tilde{\pi} =  f_M \circ \tilde{\pi} = \tilde{\f}_M =  \pi \circ \f$. Therefore there exists a scalar function $h$ with values in $S^1$ such that 
 $  \f_M  \circ \tilde{\pi} = h \f$ holds. Thus by replacing $\f$ by $h\f$ we obtain the claim.
\end{proof}

In the rest of the paper we will always assume that the diagrams just considered all commute.
Also note that we have adjusted our notation so that objects defined on $\D$ usually have 
neither a $\tilde{.}$ nor the subscript $M$.

 We want to define a (moving) frame for $f_M$ and will build it  by using the lift 
  $\f :\D \rightarrow S^5$.

On $\D$ we will use the complex coordinates  $z= x+ i y$ and $\bar z = x- i y$, respectively.
 In the following, the subscripts $z$ and $\bar z$ denote the derivatives with respect to 
 $z$ and $\bar z$, respectively, defined via  the Cauchy-Riemann operators 
\[
 \partial_z  := \frac{1}{2}\left( \frac{\partial }{\partial x} - 
 i \frac{\partial }{\partial y}\right), \quad
  \partial_{\bar z}  := \frac{1}{2}\left( \frac{\partial }{\partial x} +
 i \frac{\partial }{\partial y}\right).
\]

Thus we obtain, e.g., 
\[
 \f_z = \partial_z \f := \frac{1}{2}\left( \frac{\partial \f}{\partial x} - 
 i \frac{\partial \f}{\partial y}\right), \quad
 \f_{\bar z} = \partial_{\bar z} \f := \frac{1}{2}\left( \frac{\partial \f}{\partial x} +
 i \frac{\partial \f}{\partial y}\right).
\]

The following definition is crucial for this paper
\begin{equation} \label{xiandeta}
\xi := \f_z - (\f_z\cdot \bar{\f}) \f \hspace{3mm} \mbox{ and} \hspace{3mm} \eta := \f_{\bar z} -
(\f_{\bar z}\cdot \bar{\f}) \f. 
\end{equation}

After substitution into (\ref{Fubini-Study-sphere}) the fact that the metric $g$ is conformal gives
\begin{eqnarray} 
&&\xi \cdot \bar \eta = \xi \cdot \bar \f = \eta \cdot \bar \f = 0, \label{fconf1}\\
&& e^{- \omega} \xi \cdot \bar \xi + e^{-\omega} \eta \cdot \bar \eta := a + b = 2,   \label{fconf2}
\end{eqnarray}
where we define 
\begin{equation}
a :=e^{-\omega} \xi \cdot \bar \xi \hspace{3mm} \mbox{ and} \hspace{3mm}  b := e^{-\omega}\eta \cdot \bar \eta .
\end{equation}
Writing temporarily $\xi = \xi[\f]$ and analogously for $\eta$ it is easy to check that 
$ \xi[ h\f]= h \xi[\f]$ and  $\eta[h\f] = h\eta[\f]$ hold for all functions $h:\D \rightarrow S^1$. Therefore 
$a$ and $b$ are independent of the choices of the local lift $\f$ and the complex 
coordinate $z$. Since $0 \leq a,b \leq 2$, we can define globally an invariant 
function $\theta : \D \rightarrow [0, \pi]$ by
\begin{equation}
\theta := 2 \arccos ( \sqrt{ \frac{a}{2}} ).
\end{equation}
It is easy to verify that the invariant $\theta$ defined above is exactly the K\"ahler angle of $f$, see for example 
\cite{Wang}.
In particular, $a = b $ is equivalent to $f$ being Lagrangian. In this case $a = b = 1$.
\begin{Definition}
A point $p \in \D$ is called holomorphic (anti-holomorphic or real respectively) for
$f: \D \rightarrow \C P^2$ if $\theta(p) = 0$ ($\pi$ or $\frac{\pi}{2}$ respectively).
A point is called a complex point of $f$, if it is holomorphic or anti-holomorphic.
\end{Definition}
Note that $p$ is a complex point of $f$ if and only if $a=0$ or $a = 2$.
As a consequence, $\xi=0$ or $\eta=0$, respectively.

In order to be able to describe all immersions into $\C P^2$ without complex points we introduce two more invariants:
\begin{align}
\Phi &:= e^{-\omega} \xi_{\bar z} \cdot \bar \eta dz := \phi dz,\\
\Psi &:=  \xi_{z} \cdot \bar \eta dz^3 := \psi dz^3.
\end{align}

Using (\ref{fconf1}) one can easily check that $\Phi$ and $\Psi$ are independent of the choices of a lift $\f$ of $f$ and the complex coordinate $z$ of $\D$ and thus are globally defined on the Riemann surface $\D$.

Moreover, if $f$ is transformed by an isometry $T \in \SU$ to $Tf$, then $\xi$ and $\eta$ are transformed by $T$ to $T\xi$ and $T\eta$ respectively.
 From the definitions it follows that $\Phi$ and $\Psi$ are 
$\SU$-invariant.

We call $\Psi$ the {\it cubic Hopf differential} and $\Xi = i(\Phi - \bar \Phi)$
the \emph{mean curvature form}. (Note, some authors call $\Phi$ the mean curvature form.)

\begin{Remark}
The definitions of $\Phi$ and $\Psi$ show that the complex points  of $f$ are zeros of 
$\Phi$ and $\Psi$. 
\end{Remark}


\subsection{The moving frame equations for surfaces without complex points}
 $f:\D \rightarrow \C P^2$ be a contractible surface without complex points and let 
$\f : \D \rightarrow S^5$ be a  lift of $f$. We define $\xi$ and $\eta$ by 
 (\ref{xiandeta}).
Then at each point of $\D$ we obtain a basis of $\C^3$ given by $\{\xi, \eta, \f \}$.

We combine these vectors to form a matrix
\[
\tilde{\mathcal{F}} = ( \xi, \eta,\f).
\]
Due to \eqref{fconf1}, \eqref{fconf2} and the fact that $\f \cdot \bar \f = 1$ holds, this matrix satisfies the two equations
\[ 
\tilde{\mathcal{F}}_z =  \tilde{\mathcal{F}} \tilde{\mathcal{U}}, \quad 
\tilde{\mathcal{F}}_{\bar z} =  \tilde{\mathcal{F}} \tilde{\mathcal{V}},
\]
where
\begin{align*}
 \tilde{\mathcal U} &= \begin{pmatrix} 
 a_z/a + \omega_z + \rho + a^{-1} \phi & -a^{-1} \bar{\phi} & 1\\
  b^{-1} e^{-\omega} \psi & \rho + b^{-1} \phi & 0\\
 0 & -be^\omega & \rho
 \end{pmatrix},  \;\;
  \tilde{\mathcal V} &=
 \begin{pmatrix} 
 - \bar \rho - a	^{-1} \bar \phi
 & -a^{-1} e^{-\omega} \bar \psi  & 0 \\
 b^{-1} \phi & b_{\bar z} /b + \omega_{\bar z} - \bar \rho -b^{-1} \bar{\phi}& 1\\
 - a e^\omega & 0 & - \bar \rho
\end{pmatrix}, 
\end{align*}
and where we have abbreviated
\begin{equation}\label{def:rho}
\rho = \f_z \cdot \bar \f.
\end{equation}
Note, if we can choose $\f$ as a \emph{horizontal lift}, i.e., satisfying  $d\f \cdot \bar{f} = 0$, then $\rho = 0$.

The integrability condition 
$$ \tilde{\mathcal U}_{\bar z} -  \tilde{\mathcal V}_z = [ \tilde{\mathcal U},  \tilde{\mathcal V}]$$
splits into the following four scalar conditions
\begin{align} 
&\rho_{\bar{z}} + {\bar{\rho}}_z = (a-b) e^\omega, \label{comp1}\\
&(\log a)_{z \bar{z}} + \omega_{z \bar{z}} = (b-2a) e^\omega - (a^{-1} \phi)_{\bar{z}} -
(a^{-1} \bar{\phi})_{z} - (ab)^{-1}|\phi|^2 +(ab)^{-1} e^{-2\omega}|\psi|^2 ,  \label{comp2}\\
&\psi_{\bar z} + (a^{-1} - b^{-1} ) \bar{\phi} \psi  +  (a^{-1} - b^{-1} ) e^\omega \phi^2 = e^\omega (\phi_z - \omega_z \phi) - e^\omega \phi (\log (ab))_z, \label{comp3}\\
&(\log b)_{z \bar{z}} + \omega_{z \bar{z}} = (a-2b) e^\omega + (b^{-1} \phi)_{\bar{z}} +
(b^{-1} \bar{\phi})_{z} - (ab)^{-1}|\phi|^2 +
 (ab)^{-1} e^{-2\omega}|\psi|^2.\label{comp4}
\end{align}

We note that (\ref{comp1}) is not essential.
By using the Dolbeault Lemma, see \cite[Theorem 13.2]{Forster}, we obtain

\begin{Proposition} \label{rho=rho0}
One can choose  without loss of generality $\f$ such that $\rho$ satisfies 
\begin{equation}
\partial_{\bar{z}} \rho = \frac{1}{2}  (a-b) e^\omega.
\end{equation}
\end{Proposition}

In this case we write $\rho_0$ instead of $\rho$.

\begin{proof}
Define $\rho_0$ as above. Then (\ref{comp1}) is equivalent to
\begin{equation}
(\rho - \rho_0)_{\bar{z}} + \overline{  (\rho - \rho_0)_{ \bar z }   } = 0.
\end{equation}
Moreover, $\Omega = i \{ (\rho - \rho_0)dz - \overline{(\rho - \rho_0)}d\bar z \}$
is a closed real 1-form.
Let $\delta : \D \rightarrow \R$ denote a solution to $d \delta = \Omega.$ 
Then the new lift $\tilde{\f} = e^{i \delta} \f$ satisfies the compatibility conditions above with 
$\tilde{\rho} = \rho_0$.
\end{proof}

\begin{Remark}
\mbox{}
\begin{enumerate}
 \item The result above hinges heavily on the fact, that multiplication of $\f$ by a scalar function with values in $S^1$ does not change $a, b, \omega, \phi$ and $\psi$, as was pointed out in a previous subsection.

\item  But we will also need that the diagrams in section \ref{susc:1.3} all commute.
To maintain this fact we will also need to adjust $\f_M$ by the same factor 
we have used for $\f$.
\end{enumerate}
\end{Remark}
For later use it will be convenient to bring the matrices $\tilde{\mathcal{U}}$  and 
$\tilde{\mathcal{V}}$  into a more \emph{symmetric} form.

For this purpose we consider
\begin{equation*} \label{gaugetoF}
\mathcal{F} = \tilde{\mathcal{F}} \mathcal{R},
\end{equation*}
where $\mathcal{R}$ denotes the diagonal matrix
\begin{equation*} \label{defR}
\mathcal{R} = \di( -ie^{-\frac{\omega}{2}} \sqrt{a}^{-1}, -ie^{-\frac{\omega}{2}} \sqrt{b}^{-1},1).
\end{equation*}
Then we obtain 
\begin{equation*} \label{MC1}
\mathcal{F}^{-1} d \mathcal{F} = \mathcal{U} dz + \mathcal{V} d \bar z,
\end{equation*}
where 
\begin{equation} \label{MCU}
 \mathcal U = \begin{pmatrix} 
 \frac{1}{2} \frac{a_z}{a} +  \frac{1}{2} \omega_z + \rho + a^{-1} \phi & -\sqrt{ab}^{-1}\bar{\phi} & 
 i\sqrt{a} e^{\frac{\omega}{2}}\\
  \sqrt{ab}^{-1} e^{-\omega} \psi &    - \frac{1}{2} \frac{b_z}{b} - \frac{1}{2} \omega_z +\rho+ b^{-1} \phi & 0\\
 0 & i\sqrt{b}e^{ \frac{\omega}{2}  }& \rho
 \end{pmatrix}, 
 \end{equation}

 \begin{equation} \label{MCV}
 \mathcal V =
 \begin{pmatrix} 
 - \frac{1}{2} \frac{a_{\bar z}}{a} -  \frac{1}{2} \omega_{\bar z} -\bar{ \rho} - a^{-1} \bar{\phi} & 
 -\sqrt{ab}^{-1} e^{-\omega} \bar{\psi}
 & 0 \\
 \sqrt{ab}^{-1} \phi &  \frac{1}{2} \frac{b_{\bar{z}}}{b}  + \frac{1}{2} \omega_{\bar{z}} -\bar{\rho}   - b^{-1} \bar{\phi}& i \sqrt{b} e^{\frac{\omega}{2} } \\
 i\sqrt{a} e^{\frac{\omega}{2} }  & 0 & - \bar \rho
\end{pmatrix}.
\end{equation}

Several  remarks are in place:

\begin{enumerate}
\item $\tilde{\mathcal{F}}$ and $\mathcal{F} = \tilde{\mathcal{F}} \mathcal{R}$ have the same last column.
As a consequence $\f = \tilde{\mathcal{F}} e_3 = \mathcal{F} e_3$. Hence a local lift of some immersion $f$ without complex points can be retrieved from both frames.

\item $\mathcal{V} = - \bar{\mathcal{U}}^T$.

\item $\tr (\mathcal{U} )= (a^{-1} + b^{-1}) \phi + 3 \rho + \frac{1}{2} (\frac{a_z}{a} -\frac{b_z}{b} )$.

\item $\tr (\mathcal{V} ) = - \overline{\tr (\mathcal{U} )} $.

\item 
$\mathcal{F}^{-1} d \mathcal{F} = \mathcal{U} dz + \mathcal{V} d \bar z$ is skew-hermitian.

\item If the initial condition for the solution to the equation $\mathcal{F}^{-1} d \mathcal{F} = \mathcal{U} dz + \mathcal{V} d \bar z$ is in $\U$, then the whole solution 
$\mathcal{F}$ is in $\U$. In particular, in this case $\det(\mathcal{F}) \in S^1.$

\item Sometimes we will need to know from which $\f$ the frame $\mathcal{F}$ has been constructed. In such a case we write  $\mathcal{F} =  \mathcal{F} (\f)$.
\end{enumerate}

With these pieces of information we obtain
\begin{Theorem}[Fundamental Theorem of liftable surfaces in $\C P^2$ without complex points, \cite{LW}] \label{FTL}
\mbox{}
\begin{enumerate}
 \item  Let $f_M:M \rightarrow \C P^2$ be a liftable  immersion without complex points and lift $\f_M: M \rightarrow S^5$. Let $g = 2 e^\omega dz d\bar{z}$ denote the induced metric, $\theta : \D \rightarrow (0, \pi )$ the K\"ahler angle, $\Psi$ the cubic form, $\Phi$ the mean curvature form (all defined from $\f$ as in section $\ref{susc:1.3}$). 
Set $a=2\cos\frac{\theta}{2}$, $b=2\sin\frac{\theta}{2}$ and $\rho=\rho_0$ is defined by \eqref{def:rho}.
Then  the conditions  \eqref{comp1}, 
\eqref{comp2},  \eqref{comp3} and   \eqref{comp4} are satisfied.

\item Conversely, let $g = 2 e^\omega dz d\bar{z}$ be a Riemannian metric on the simply connected Riemann surface $\D$. Let $\theta : \D \rightarrow (0,\pi)$ be a real valued function and $\Phi$ and $\Psi$ a $(1,0)$-form and a $(3,0)$-form  on $\D$ respectively. 
Set $a=2\cos\frac{\theta}{2}$, $b=2\sin\frac{\theta}{2}$ and $\rho=\rho_0$ is given in Proposition $\ref{rho=rho0}$.
If these data satisfy the conditions  \eqref{comp1},  \eqref{comp2},  \eqref{comp3} and   \eqref{comp4}, then there exists an immersion $f:\D \rightarrow \C P^2$ without complex points such that the given data have the meaning for $f$ as stated in $(1)$.
In particular, in this case the given data are the corresponding invariants for some lift $\f : \D \rightarrow S^5$ of $f$.

Moreover, if $\f$ descends to a map $\hat{\f}: \hat{M} \rightarrow S^5$ for some Riemann surface  $ \hat{M}$, then $\pi \circ \hat{\f} : \hat{M} \rightarrow \C P^2$  is a liftable immersion without complex points from $ \hat{M}$ to $\C P^2$.

\item  If two isometric surfaces $f: \D\rightarrow \C P^2$ and $\hat{f} : \hat{\D}\rightarrow \C P^2$ without complex points have the same forms $\Phi$ and $\Psi$ and the same K\"ahler function $\theta$, i.e., if there exists some diffeomorphism $\kappa : \D \rightarrow \hat{\D}$ such that $g = \kappa^* \hat{g}$, $\theta = \hat{\theta} \circ \kappa$, $\Phi = \kappa^* \hat\Phi$ and $\Psi = \kappa^*{\hat\Psi}$, then there exists 
an isometry $T \in \SU$ such that $f = T \circ \hat{f} \circ \kappa$.
\end{enumerate}
\end{Theorem}
The following result will be particularly convenient.
\begin{Theorem}\label{det=1}
Let $f_M:M \rightarrow \C P^2$ be a liftable  immersion without complex points.
Then one can choose, without loss of generality, a  lift $\f_M: M \rightarrow S^5$ of $f$ such that  the corresponding frame $\mathcal{F}$ has determinant $1$. Such lift is called a special lift for $f_M$.
\end{Theorem}

\begin{proof}
First we note  that $\mathcal{F}$ is unitary and thus has determinant $\delta$ in $S^1$. Since $\mathcal{F}$ 
is defined on $\D$, we can take a cubic root of $\delta$. Let $h\in S^1$ denote the inverse of this cubic root. Then the property 
$ \xi[ h\f]= h \xi[\f]$ and  $\eta[h\f] = h\eta[\f]$ for all functions $h:\D \rightarrow S^1$ implies the claim, where we also need to adjust $\f_M$ as before. 
\end{proof}

\begin{Corollary}
Under the assumptions above one can specialize the lift $\f_M$ in two ways (by multiplication by a function with values in $S^1$) so that one can assume that $\rho = \rho_0$ holds, or so that $\mathcal{F} \in \SU$ holds.
\end{Corollary}

From here on we will always assume that $\mathcal{F} \in \SU$.
It is important to note that now $\rho$ can not be assumed to have the special form $\rho_0$.

\begin{Corollary}
Under the assumptions above the trace of the Maurer-Cartan form of $\mathcal{F}$ vanishes identically.
\end{Corollary}
In the case of minimal Lagrangian surfaces, $a=b=1$ and $\Phi\equiv0$, thus one can have a horizontal lift $\f$, i.e., $\rho=0$ and the trace of the Maurer-Cartan form of $\mathcal{F}$ vanishes automatically.


\section{Algebraic digression}\label{sc:alg}
In this section we discuss briefly the algebraic setting which will be used in the following sections.
\subsection{The automorphism $\sigma$}
We consider the order $6$ automorphism  $\sigma$ of the Lie algebra  
$\mathfrak{g} =  \sl$,
given by 
\begin{equation} \label{defsigma}
\sigma ( X) =   - P X^T P,  \quad \mbox{where} \quad 
P = 
\begin{pmatrix}
0& \epsilon^2 & 0\\
\epsilon^4& 0 & 0\\
0&0&1\\
\end{pmatrix}
\end{equation}
with $\epsilon = e^{\frac{ i\pi}{3}}$.
We also consider the connected Lie group $G =  \SL$ and the automorphism $\sigma^G$ of order $6$
\begin{equation}\label{eq:sigmaonG}
\sigma^G(g) =  P (g^{T})^{-1}P.
\end{equation}
Then $\sigma$ is the differential of $\sigma^G$ and we will use from now on the same notation for both homomorphisms.
 By abuse of notation we will also write $\sigma^G$ by $\sigma$.
We would like to point out that $\sigma$ is an outer automorphism of $\mathfrak{g}$.

By a simple computation  we obtain
\begin{equation}\label{eq:sigma2}
\sigma^2 (X) = P_2 X P_2^{-1},\quad \mbox{with} \quad 
 P_2 = \di ( \epsilon^4, \epsilon^2, 1).
\end{equation}
The formula on the group level is the same.
We would like to point out that $\sigma^2$ is an inner automorphism.

Finally we derive
\begin{equation}\label{eq:sigma3}
\sigma^3 ( X) =  -P_3 X^T P_3,
\quad \mbox{with} \quad 
P _3= 
\begin{pmatrix}
0& 1 & 0\\
1& 0 & 0\\
0&0&1\\
\end{pmatrix}.
\end{equation}
We would like to point out that $\sigma^3$ is an outer automorphism of $\mathfrak{g}$.
In the context of certain  surface classes  we will discuss the automorphisms 
$\sigma$, $\sigma^2$ and $\sigma^3$ of $\mathfrak{g}$.

Since, the eigenspaces of the various powers of $\sigma$ all can be derived from the eigenspaces of $\sigma$, we discuss this case first.
Explicitly the eigenspaces $\mathfrak{g}_k$ of $\sigma$ with respect to the eigenvalue $\epsilon^k$ in $\sl$ are given as follows
\begin{align*}
\mathfrak{g}_0&=\left\{
                    \begin{pmatrix}
                    a &  &  \\
                     & -a &  \\
                     &  & 0 \\
                \end{pmatrix}
                \mid a\in \C
\right\},
\quad
\mathfrak{g}_1=\left\{\begin{pmatrix}
                    0 & b & 0 \\
                     0 & 0 & a \\
                     a & 0 & 0 \\
                  \end{pmatrix}\mid a,b\in\C
\right\},
\\
\mathfrak{g}_2&=\left\{
                  \begin{pmatrix}
                    0 & 0& a  \\
                     0& 0 & 0 \\
                     0& -a & 0 \\
                  \end{pmatrix} \mid a\in \C
              \right\},
\quad
\mathfrak{g}_3=\left\{
                  \begin{pmatrix}
                    a &  &  \\
                      & a &  \\
                      &  & -2a \\
                  \end{pmatrix} \mid a\in \C
              \right\},\\
\mathfrak{g}_4&=\left\{
                  \begin{pmatrix}
                    0 & 0 & 0  \\
                     0& 0 & a \\
                     -a & 0 & 0 \\
                  \end{pmatrix} \mid a\in \C
            \right\},
\quad
\mathfrak{g}_5=\left\{
                  \begin{pmatrix}
                    0 & 0 & a \\
                     b & 0 & 0 \\
                     0 & a & 0 \\
                  \end{pmatrix} \mid a, b\in \C 
                \right\}.
\end{align*}

The eigenspaces of $\sigma^2$ are 
\begin{itemize}
\item $\mathfrak{g}_1 + \mathfrak{g}_4$ for the eigenvalue  $\epsilon^2$,
\item  $\mathfrak{g}_2 + \mathfrak{g}_5$ for the eigenvalue  $\epsilon^4$,
\item $\mathfrak{g}_3 + \mathfrak{g}_0$ for the eigenvalue  $1 $.
\end{itemize}
The eigenspaces for $\sigma^3$ are
\begin{itemize}
\item $\mathfrak{g}_4 + \mathfrak{g}_2 + \mathfrak{g}_0$ for the eigenvalue  $1$,
\item $\mathfrak{g}_1 + \mathfrak{g}_3 + \mathfrak{g}_5$ for the eigenvalue  $\epsilon^3 =  -1 $.
\end{itemize}

\subsection{The real form involution $\tau$}\label{subsc:realform}
 The real form $\mathfrak g^{\R}=\su$ of $\mathfrak g = \sl$ is given by the anti-linear involution 
\begin{equation*}\label{eq:tau}
 \tau (X) = - \bar X^T, \quad X \in \sl.
\end{equation*}
 We also consider the anti-linear involution $\tau^G$ on $G = \SL$ 
\begin{equation}\label{eq:tauonG0}
 \tau^G (g) = (\bar g^T)^{-1}, \quad g \in \SL.
\end{equation} 
 By abuse of notation we will also write $\tau^{G}$ by $\tau$.
 Then a direct computation shows that $\sigma$ in \eqref{defsigma}
 and $\tau$ commute, i.e.,
$\tau \circ \sigma =  \sigma \circ \tau$,
 thus $\tau$ and the eigenspaces of $\sigma$ have the relation 
\begin{equation}\label{eq:tausigma}
\tau (\mathfrak g_j) = \mathfrak g_{-j}, \quad j=0, 1, \dots, 5.
\end{equation}
 In particular $\mathfrak g_0$ and $\mathfrak g_0 \oplus \mathfrak g_3$ are subalgebras of $\mathfrak {su}_3$ with the obvious complexifications.

\subsection{$k$-symmetric spaces}
 As we discussed in section \ref{subsc:realform}, $\sigma$ and $\tau$
 commute, and thus we arrive at a definition of $k$-symmetric  spaces, as it will be used in our paper.
 
 \begin{Definition}\label{def:k-symmetric}
 Let $G^{\R}/G^{\R}_0$ be a real homogeneous space such that $G^{\R}$ is a 
 real form of a complex Lie group $G$ given by a real form involution 
 $\tau$, that is, $G^{\R} = \Fi (G, \tau)$.
 Moreover,  let $\sigma$ 
 be an order $k \;(k\geq 2)$ automorphism  of $G,$ leaving  $G^{\R}$ invariant and commuting with $\tau$.
 Then  $G^{\R}/G^{\R}_0$ is called a \emph{$k$-symmetric space} if 
 the following condition is satisfied
\begin{equation}
 \Fi(G^{\R},\sigma)^\circ 
 \subset G^{\R}_0 \subset \Fi (G^{\R}, \sigma),
\end{equation}
 where $\Fi(G^{\R}, \sigma)^\circ$ denotes the identity component of $\Fi(G^{\R}, \sigma)$.
\end{Definition}

\subsection{Primitive maps and the loop group formalism}
In the last subsection we have discussed complex Lie algebras and Lie groups.
For the applications to geometry we will need to work with real Lie groups.

Thus we consider a complex Lie group as before and let $\tau$ denote an 
anti-holomorphic involution of $G$. Then we put
\begin{equation*}
G^\R = \Fi(G,\tau).
\end{equation*}

Similarly we define $\operatorname{Lie}G^\R = \mathfrak{g}^\R.$
\begin{Definition}
Let $\kappa$ be any automorphism of $\mathfrak{g}$ of finite order $k > 2$.
Let $\mathfrak{g}_m$ denote the  eigenspaces of $\kappa$, where we choose $m \in \mathbb{Z}$ and actually work with $m \mod k$.
Let $\mathcal{F} : \D \rightarrow G$ be a smooth map. Then $\mathcal{F}$ will be called \emph{primitive relative to $\kappa$} if
\begin{equation*}
\mathcal{F}^{-1} d\mathcal{F} = \alpha_{-1}  dz + \alpha_0^{\prime} dz+\alpha_0^{\prime\prime} d\bar{z}+ \alpha_1 d\bar{z} \in \mathfrak{g}_{-1} +  \mathfrak{g}_0 + \mathfrak{g}_1,
\end{equation*}
where  $\alpha_m$, $\alpha_0^{\prime}$ and $\alpha_0^{\prime\prime}$ take values in  an eigenspace  
$\mathfrak{g}_m$ of  $\kappa$.
\end{Definition}
By abuse of notation we will also write $\alpha_0 =  \alpha_0^\prime dz + \alpha_0^{\prime \prime}d \bar{z}$.
\begin{Lemma}
Let $\mathcal{F}$  be primitive relative to $\kappa$ and let us write
$\mathcal{F}^{-1} d\mathcal{F} = \alpha_{-1} dz+ \alpha_0 + \alpha_1 d \bar{z}$.
Then
$\lambda^{-1} \alpha_{-1} dz+ \alpha_0 + \lambda \alpha_1d \bar{z}$
is integrable for all $\lambda \in \C^*$.
\end{Lemma}
\begin{proof}
Together with a  straightforward computation one needs to use that because of $k>2$ the sum $\mathfrak{g}_{-1} + \mathfrak{g}_0 + \mathfrak{g}_1$ of eigenspaces is direct.
\end{proof}
The importance of this observation has been elaborated on and explained in
\cite[Section 3.2]{BP} and \cite{B}.
\begin{Theorem}[\cite{BP, B}] \label{theoremharmonicity}
Let $G$ be a complex Lie group,   $\sigma$ an automorphism of $G$ of finite order $k \geq 2$
and $\tau$ an anti-holomorphic involution of $G$ which commutes with  $\sigma$.
Let $G^\R_0$ be any Lie subgroup of $G^\R$ satisfying 
 $\Fi (G^\R, \sigma)^{\circ} \subset {G^\R_0} \subset \Fi (G^\R, \sigma)$.
Then we consider the $k$-symmetric space $G^\R / {G^\R_0}$ together with the (pseudo-)Riemannian structure induced by some bi-invariant (pseudo-)Riemannian structure on $G^\R$. 
Let $h:\D \rightarrow G^\R / {G^\R_0}$ be a 
smooth map and 
$\mathcal{F} : \D \rightarrow G^\R$ a frame for $h$, i.e., $h = \pi \circ \mathcal{F}$, where
$\pi: G^\R \rightarrow G^{\R}/{G^\R_0}$ denotes the canonical projection.

Then the following statements hold:
\begin{enumerate}
 \item   If $k=2$, then $h$ is harmonic if and only if  $\lambda^{-1} \alpha_{-1} dz+ \alpha_0 + \lambda \alpha_1 d\bar{z}$ is integrable for all $\lambda \in \C^*$.

\item If $k>2$, then $h$ is harmonic  if $\mathcal{F}$ is primitive relative 
to  $ \sigma$.
\end{enumerate}
\end{Theorem}
 From the above theorem, we have the following definition.
\begin{Definition}
 Retain the notation in Theorem \ref{theoremharmonicity}.
\begin{enumerate}
 \item 
 The frame $\mathcal F$ is called \emph{primitive harmonic}, if
 $\mathcal{F}^{-1} d\mathcal{F} = 
 \alpha_{-1}dz  + \alpha_0 + \alpha_1 d \bar{z} $ such that 
 $\lambda^{-1} \alpha_{-1} dz+ \alpha_0 + \lambda \alpha_1 d\bar{z}$ 
 is integrable for all $\lambda \in \C^*$.
\item  The map $h$ is called 
 \emph{primitive harmonic map}, if the frame $\mathcal F$ is 
 primitive harmonic.
\end{enumerate}
\end{Definition}

This admits a direct application of the \emph{loop group method} (see \cite{DPW} for the basic formalism, presented in loc.cit. for compact groups.)

The first step here is to integrate
\begin{equation*}
\mathcal{F_\lambda}^{-1} d\mathcal{F_\lambda} = 
\lambda^{-1} \alpha_{-1} dz+ \alpha_0 + \lambda \alpha_1d \bar{z}.
\end{equation*}

Since $\tau$ maps $\mathfrak{g}_m$ to $\mathfrak{g}_{-m}$, we can assume that
$\mathcal{F}_\lambda$ is contained in $G^\R$ for all $\lambda \in S^1$.
Note that we will write $\mathcal{F}(z,\lambda)$ or  $\mathcal{F}_\lambda(z)$, whatever is most convenient.
We will usually also assume $\mathcal{F}(z_0,\lambda) = I$ for a once and for all fixed base point $z_0$.

Then it follows from the above that also $h_\lambda = \mathcal{F}_\lambda \mod G^\R_0$ is a primitive harmonic map with frame $\mathcal{F}_\lambda$.
Usually  $\mathcal{F}_\lambda$ is called \emph{an extended frame} for $h$.

The loop group method constructs in principle all these extended frames.
For this one does not read $\mathcal{F}(z,\lambda) $ as a family of frames, 
parametrized by $\lambda$, but as a function of $z$ into some loop group.

Here are the basic definitions:

\begin{enumerate}
\item $\Lambda G = \{ g: S^1 \rightarrow G \}.$

\noindent Considering $G$ as a subgroup of some matrix algebra $\operatorname{Mat} (n,\C)$  we use  the Wiener norm on $\Lambda \operatorname{Mat} (n,\C)$ and thus induce a Banach Lie group structure on $ \Lambda G $.

\item $\Lambda^+G = \left\{g \in G\mid 
 \begin{array}{l}
  \mbox{$g$ has  a holomorphic extension to the open unit disk}\\
  \mbox{and $g^{-1}$ has the same property}
 \end{array} 
 \right\}$.

\item $\Lambda_*^+G = \{ g \in  \Lambda^+G \mid g(0) = I \}.$

\item $\Lambda^-G = \left\{g \in G \mid  \begin{array}{l}
 \mbox{$g$ has  a holomorphic 
extension to the open  upper unit disk in $\C P^1$}  \\

 \mbox{and  $g^{-1}$ has the same property} 
\end{array}
\right\}.$

\item $\Lambda_*^-G = \{ g \in  \Lambda^-G \mid g(\infty) = I \}.$

\item  $\Lambda G^\R = \{g \in \Lambda G\mid \tau (g(\lambda) )= g(\lambda) \mbox{ for all} \hspace{1mm} \lambda \in S^1\}$.
\end{enumerate}
Finally, we will actually always use \emph{twisted subgroups} of the groups above.
First we have
\[
\Lambda G_{\sigma}= 
\left\{ 
 g \in \Lambda G\mid  
\mbox{$\sigma( g( \epsilon^{- 1} \lambda)) = g(\lambda)$  for all $\lambda \in S^1$} 
\right\}.
\]
The other twisted groups are defined analogously, like 
\[
 \Lambda_*^+G _{\sigma} = 
\Lambda_*^+G  \cap \Lambda G_{\sigma}.
\]
By the form of  $\mathcal{F_\lambda}^{-1} d\mathcal{F_\lambda}$ we infer that all the loop matrices associated with geometric quantities are actually defined for all $\lambda \in \C^*$.
However, geometric interpretations are usually only possible for $\lambda \in S^1$.

To understand the construction procedure mentioned above one considers next
again $h$ and $\mathcal{F}$ as above and  decomposes 
\begin{equation*}
\mathcal{F}(z,\lambda) = \mathcal{C}(z,\lambda) \cdot \mathcal{L}_+(z,\lambda),
 \end{equation*}
where $\mathcal{C}$ is holomorphic in $z \in \D$ and holomorphic in $\lambda \in \C^*$
and $\mathcal{L}_+(z,\lambda) \in \Lambda ^+G _{\sigma} $.

Since $S^2$ does not occur in this paper as domain of a harmonic map, such a decomposition is always possible, 
and defines \emph{a holomorphic potential $\eta$ for $h$} by the formula
\[
\eta = \mathcal{C}^{-1} d\mathcal{C}. 
\]
The potential $\eta$ takes the form
\begin{equation} \label{eta}
\eta = \lambda^{-1} \eta_{-1}(z) dz + \lambda^{0} \eta_{0}(z) dz +\lambda^{1} \eta_{1}(z) dz
+\lambda^{2} \eta_{2}(z) dz + \cdots.
\end{equation}
We would like to emphasize:
\begin{enumerate}
\item All coefficient functions $\eta_j$ are holomorphic on $\D$.
\item All  $\eta_j$ are contained in $\mathfrak{g}_{m(j)}$, where $m(j) = 0,1,2,\dots,k-1$ and $m(j) \equiv j \mod k$.
\end{enumerate}
This explains the procedure to obtain a holomorphic potential from a primitive harmonic map.
The fortunate point is that this procedure can be reversed.

The following theorem  is a straightforward generalization of a result of  \cite{DPW}.
\begin{Theorem} [Loop group procedure, \cite{DPW}]
Let $G, \sigma$ and $ \tau $ as above.
 Let  $h:\D \rightarrow G^\R / {G^\R_0}$ 
be a primitive harmonic map with extended frame $\mathcal{F}_\lambda$.
Define $\mathcal{C}$ by $\mathcal{F}(z,\lambda) = \mathcal{C}(z,\lambda) \cdot \mathcal{L}_+(z,\lambda)$ and put $\eta = \mathcal{C}^{-1} d\mathcal{C}$.
Then $\eta$ has the form stated in \eqref{eta}, the coefficient functions $\eta_j$ of $\eta$ are holomorphic on $\D$ and we have 
$\eta_j \in \mathfrak{g}_{m(j)}$ and $m(j) \equiv j \mod k$.

Conversely, consider any holomorphic 1-form $\xi$ satisfying the three conditions just listed for $\eta$.
Then solve the ODE $d\mathcal{C} = \mathcal{C} \xi$ on $\D$ with $\mathcal{C} \in \Lambda G$.
Next write $\mathcal{C} = \mathcal{F} _\lambda \cdot \mathcal{V}_+$
with $\mathcal{F}_\lambda  \in \Lambda G^\R_{\sigma}$ and $\mathcal{V}_+ \in 
 \Lambda ^+G _{\sigma} $.
Then $\mathcal{F}_\lambda$ is the extended frame of the associated family of some primitive harmonic map $h : \D \rightarrow G^\R / {G^\R_0}.$
\end{Theorem}

\subsection{Evaluation of the meaning of primitive harmonic maps relative 
to $\sigma$, $\sigma	^2$, and $\sigma^3$}
We start by  evaluating what it means that a Maurer-Cartan form
of some frame of some liftable immersion into $\C P^2$ without complex points 
is primitive relative to  $\sigma$, $\sigma^2$, or $\sigma^3$ respectively.

We will use the notation introduced just above.
\begin{Theorem} \label{equivprimitive}
Let $G = \SL$ and $\mathfrak{g} = \sl$ its Lie algebra.
Let $\tau$ denote the real form involution of $G$ singling out 
 $G^{\R}=\SU$ in $G$
and let $\sigma = \sigma^G$ be the automorphism of order $6$ of $G$ given by 
$\sigma (g) =   P (g^{T})^{-1} P$  in  \eqref{defsigma}.
Assume moreover, that $\f$ is the lift of  a liftable immersion $f$ into $\C P^2$ without complex points and with frame  $\mathcal F$ in $G^{\R}$.
Then the following statements hold:
\begin{enumerate}
 \item $\mathcal F$ is primitive harmonic relative to $\sigma$ if and only if $f$ is minimal 
Lagrangian in $\C P^2$.

\item $\mathcal F$ is primitive harmonic relative to $\sigma^2$ if and only if $f$ is minimal in $\C P^2$
without complex points.

\item $\mathcal F$ is primitive harmonic relative to $\sigma^3$ if and only if either $f$ is minimal 
Lagrangian or  $f$ is flat homogeneous in $\C P^2$.
\end{enumerate}
\end{Theorem}
\begin{Remark}
 From Theorem \ref{equivprimitive}, in each case, we have 
 a primitive harmonic map in $G^{\R}/G_0^{\R}$ relative to 
 $\sigma$, $\sigma^2$ and $\sigma^3$, respectively.
 We will discuss these maps in section \ref{sc:Ruh-Vilms} in detail.
\end{Remark}
\begin{proof} Since our statement basically only uses local properties, we can assume without loss of generality that $f$ and $\f$ are defined on a contractible domain $\D$.

$(1)$ We consider the Maurer-Cartan form  $\alpha$ of some frame of $\f$. Then primitive harmonicity relative to $\sigma$ means that there is no component of $\alpha$ in the spaces 
 $\mathfrak{g}_j$, $j = 2,3,4$.
It is straightforward to see that this is equivalent to $\phi = 0$, $a=b (=1)$ and that the diagonal is in  $\mathfrak{g}_0$. In particular $\rho = 0$ and the matrices \eqref{MCU} and \eqref{MCV} have exactly the form of the Maurer-Cartan form of a minimal Lagrangian immersion (including the case $\Psi = 0$).

$(2)$  We consider again  the Maurer-Cartan form  of $\f$. Then primitive harmonicity relative to $\sigma^2$ means that there is no component of $\mathcal{U} $ in the 
spaces $\mathfrak{g}_j$, $j = 1, 4$, and there is no component of $\mathcal{V}$ in the 
spaces $\mathfrak{g}_j$, $j = 2, 5$. These two conditions are equivalent to $\phi = 0$.

 Thus the  primitivity relative to $\sigma^2$ is equivalent to $f$ being minimal without complex points by  Theorem \ref{FTL}.

$(3)$  In this case we need to consider $\mathcal U = {\mathcal U}_0  + \lambda \mathcal U_1 $ for all $\lambda \in S^1$, where $\mathcal U_0 $ takes values in the fixed point space of $\sigma^3$ and $\mathcal U_1 $ takes values in  the eigenspace for the eigenvalue $-1$.

 From section 2 we know that the eigenspaces for $\sigma^3$ are $\mathfrak{g}_4 + \mathfrak{g}_2 + \mathfrak{g}_0$ for the eigenvalue  $1$ and  $\mathfrak{g}_1 + \mathfrak{g}_3 + \mathfrak{g}_5$ for the eigenvalue  $\epsilon^3 =  -1 $.
We thus consider a primitive 1-form 
$\hat{\alpha} =  \mathcal{U} dz + \mathcal{V} d \bar z$, where  $ \mathcal{U} $ is of the form
\begin{equation*} \label{MCUsigma3}
\mathcal U =  \mathcal U_0  + \lambda^{-1} \mathcal U_1,
\end{equation*}
where
\begin{align*}
\mathcal U_0&=
\begin{pmatrix} 
 u_{11} &  0 &  \frac{i}{2} (\sqrt{a}- \sqrt{b}) e^{\frac{\omega}{2}}\\
 0  & -u_{11} & 0\\
 0 & -  \frac{i}{2} (\sqrt{a}- \sqrt{b}) e^{\frac{\omega}{2}} & 0
 \end{pmatrix}, 
 \\
\mathcal U_1&=
  \begin{pmatrix} 
 w &  - \sqrt{ab}^{-1} \phi &  \frac{i}{2} (\sqrt{a} + \sqrt{b}) e^{\frac{\omega}{2}}\\
 \sqrt{ab}^{-1} e^{-\omega} \psi  & w & 0\\
 0 &   \frac{i}{2} (\sqrt{a} + \sqrt{b}) e^{\frac{\omega}{2}} & -2w
 \end{pmatrix},
\end{align*}
 and an analogous expression holds for 
 $\mathcal{V}  = -(\overline{\mathcal{U}})^T.$

We already know from the beginning of the proof that for $\lambda = 1$ the map $\f$ comes from some immersion $f$  into $\C P^2$ without complex points.
It thus is of importance to observe that our 1-form $\alpha$ has the form stated in 
(\ref{MCU}) and  (\ref{MCV}).
As a consequence, we know the form of the diagonal entries of $\mathcal{U} $ and $\mathcal{V}$.

Next we evaluate the integrability condition  $d \alpha + \frac{1}{2} [\alpha \wedge \alpha] = 0$. Expanding in $\lambda$ it is easy to see that only the following two equations need to be evaluated:
\begin{enumerate}
\item $\partial_{\bar{z}} \mathcal{U}_1 = [ \mathcal{U}_1,  \mathcal{V}_0]$,
\item  $ \partial_{\bar{z}} \mathcal{U}_0 -  \partial_{z} \mathcal{V}_0 =
[ \mathcal{U}_1 ,\mathcal{V}_1 ] +     [  \mathcal{U}_0 , \mathcal{V}_0 ]$.
\end{enumerate}
We look first at the first of these matrix equations and evaluate the $(11)$-entry and 
the $(23)$-entry. After a simple computation we obtain $\partial_{\bar{z}} w = -\frac{1}{4} (a-b) e^\omega$ and $ -3 w \frac{i}{2} (\sqrt{a} - \sqrt{b}) e^{\frac{\omega}{2}} = 0$. Altogether we conclude that $a=b$ holds on $\D$. In particular we then also have $a=b=1$ and that $w$ is holomorphic.

We can assume, since we have normalized our frames to have determinant $1$, by using  formula $(3)$ after \eqref{MCV}, that 
$\tr (\mathcal{U} )= (a^{-1} + b^{-1}) \phi + 3 \rho + 
\frac{1}{2} (\frac{a_z}{a} -\frac{b_z}{b} )= 0, $ holds. Thus  we have  
$ \rho = - \frac{2}{3} \phi= -2 w$.  In particular, $\phi$ is  holomorphic.

Evaluating the matrix equation $(1)$ above further,  we obtain from the matrix entry $(13)$ the equation 
$\frac{\omega_z}{2} = u_{11}$. Now the equations for the matrix entries $(12)$  and $(21)$ 
imply  $\partial_{\bar{z}}\phi = 2\bar{u}_{11} \phi$ and that $\psi$ is holomorphic  respectively.
As a result, $\phi = 0$ and the $1$-form $\alpha$ is exactly the Maurer-Cartan form of the 
$\SU$-frame of a minimal Lagrangian immersion, or $\omega$, $\phi$ and $\psi$ are all constant with $\psi$ is non-vanishing, which gives a Lagrangian homogeneous surface. 
One can check easily that with these conditions  the  1-form 
$\alpha$ actually is primitive relative to $\sigma^3$.
\end{proof}
In the sections above we had always assumed that we start from some liftable immersion 
into $\C P^2$ and consider the frame constructed at the beginning of this paper.

The next theorem is more general. As before we will use  $e_3 = (0,0,1)^T$. 
\begin{Theorem}\label{thm:primitive} 
Let $\hat{\f} : \D \rightarrow S^5$ be a smooth map and 
$\hat{\mathcal{F}}:\D \rightarrow \SU$ a frame such that $\hat {\mathcal{F}}.e_3 = \hat \f$.
Moreover, we assume that the Maurer-Cartan form $\hat{\alpha}$ of $\hat{\mathcal{F}}$ has the general form, 
more precisely we thus consider a  $1$-form 
$\hat{\alpha} = \hat{ \mathcal{U}} dz + \hat{\mathcal{V}} d \bar z$,  in $\lsu$, where
\begin{equation*} \label{MCUprim}
\hat{\mathcal U} 
=
\begin{pmatrix} 
 u_{11}  + w &  u_{12} &  u_{13}\\
 u_{21}  & -u_{11}  + w & 0\\
 0 & u_{32} &  - 2 w
 \end{pmatrix}, 
 \end{equation*}
 and 
 \begin{equation*} \label{MCVprim}
\hat{\mathcal V} =
 \begin{pmatrix} 
 -\bar{ u}_{11}  -\bar{w}  & -\bar{u}_{21} & 0 \\
 -\bar{u}_{12} &   \bar{ u}_{11}  -\bar{w} &- \bar{ u}_{32}\\
  -\bar{ u}_{13} & 0 &  2 \bar{w}\\
\end{pmatrix}.
\end{equation*}
Furthermore, we assume that $u_{13}$ and $u_{23}$  never vanish.
Then the following statements hold:
\begin{enumerate}
\item Each primitive harmonic $\hat{\mathcal{F}}$  
relative to $\sigma$ can be derived from a minimal Lagrangian immersion in $\C P^2$.

\item Each primitive harmonic $\hat{\mathcal{F}}$ 
relative to $\sigma^2$ can be derived from a minimal immersion in $\C P^2$ without complex points. 

\item Each primitive harmonic $\hat{\mathcal{F}}$  relative to $\sigma^3$  can be derived from a minimal Lagrangian immersion or a flat homogeneous immersion in $\C P^2$.
\end{enumerate}
\end{Theorem}
\begin{proof} For the proof we can replace without loss of generality a given $\f$ by a gauged one.
Hence, by using the assumptions one can gauge 
 $\hat{\alpha} $ by a diagonal matrix in $\SU$ such that for the matrix entries ${u}_{jk}$ 
 of  ${\hat{\mathcal{U}}}$ we obtain: ${u}_{13} = iA$ and  ${u}_{32} = iB$  with $A$ and $B$ globally defined positive functions. We also define ${\psi}$ by  putting 
   ${u}_{21} = (AB)^{-1}{\psi}$.
   Then the matrix $\hat{\mathcal{U}}$ attains the form
   \begin{equation*} \label{hatU}
 \hat{\mathcal{U}} =
\begin{pmatrix}
  u_{11}  + w       &   u_{12}  &         i A \\
 (AB)^{-1} \psi &   -u_{11}  + w &     0\\
0         &         i B   &    -2w \\
\end{pmatrix}
\end{equation*}
and $\hat{\mathcal{V}} = -(\overline{\hat{\mathcal{U}}})^T$ has the form 
  \begin{equation*} \label{hatV}
 \hat{\mathcal{V}} =
\begin{pmatrix}
- \bar{ \hat{u}}_{11}  -\bar{w}         &      - (AB)^{-1} \bar{\psi}   &     0    \\
  - \bar{u}_{12}&  \bar{ {u}}_{11}  -\bar{w}   &     iB\\
iA        &         0   &    2 \bar{w}  \\
\end{pmatrix}.
\end{equation*}
It is not difficult to verify that there exist uniquely determined $a,b >0$ such that
\[
A = \sqrt{a}e^{\omega/2}, \quad  B = \sqrt{b}e^{\omega/2}, \quad  
a + b = 2.
\]
Using these definitions we finally  define $\phi$ by the equation:  $u_{12} = - \sqrt{ab}^{-1} \phi$.
 By assumption we have the solution $\hat{\mathcal{F}}$ to  the system
  $d \hat{\mathcal{F}} = \hat{\mathcal{F}}\hat{\alpha}$.
  We write
 $ \hat{\mathcal{F}} = ( \hat{\f}_1, \hat{\f}_2, \hat{\f}).$
 An evaluation of the equation for $\hat{\mathcal{F}}$ yields
 \begin{equation*}
 \hat\f_z=  i A  \hat{ \f }_1  -2w \hat{\f}.
 \end{equation*}
 
 And similarly we obtain for  ${\hat \f}_{\bar{z}} $ in view of 
 $\hat{\mathcal{V}}  = -(\overline{\hat{\mathcal{U}}})^T$  the equation
    \begin{equation*}
{\hat \f}_{\bar{z}}=  i B  {\hat \f }_2   + 2 \bar{w} {\hat \f}.
 \end{equation*}

A simple calculation shows now that  ${\hat \f }_z$ and $  { \hat \f }_{\bar{z}}$ are linearly independent everywhere. Thus $\hat\f$ is an immersion (into $S^5$) and it follows that the projection $\hat f$ of $\hat\f$ to $\C P^2$ is also an immersion
 (and then obviously has the lift $\hat \f$). We need to show that $\hat f$ does not have any complex points.
 For this we consider as in (\ref{xiandeta}):
 \begin{equation*} \label{xiandeta2}
\xi :=\hat \f_z - (\hat \f_z\cdot \bar{\hat \f}){\hat \f} \hspace{3mm} \mbox{ and} \hspace{3mm} \eta := {\hat \f}_{\bar z} -
({\hat \f}_{\bar z}\cdot \bar{\hat \f}) \hat \f. 
\end{equation*}

A straightforward computation yields $\xi \cdot \bar{\xi} = A^2 >0$ and a similar computation yields $\eta \cdot \bar{\eta} = B^2 >0$.
Next we observe that one can write $A$ and $B$ uniquely in the form $A = \sqrt{a} 
 e^{\frac{\omega}{2}}$, $B = \sqrt{b} e^{\frac{\omega}{2}}$, with $a,b >0$ and $a+b = 2$.
Then we can rewrite $ a = e^{-\omega} \xi \cdot \bar{\xi}$ and similarly 
$ b = e^{-\omega} \eta \cdot \bar{\eta}$ and it follows that $f$ does not have any complex points.

Note that this implies that $\hat f$ induces the metric  $g = 2 e^\omega dz d\bar{z}$ (by comparison to section \ref{susc:1.2}).
Moreover we infer $\hat{\rho}= {\hat \f}_z \cdot \bar{\hat \f}$ and also 
$\hat{\f}_1 = -ie^{-\frac{\omega}{2}} \sqrt{a}^{-1} \xi$. In addition, this implies $\hat{\rho} = -2w$.
Similarly, we have $\hat{\f}_2 = -ie^{-\frac{\omega}{2}} \sqrt{b}^{-1} \eta$.

As a consequence, the frame $\hat{\mathcal{F}}$ coincides with the frame $\mathcal{F}$. 
Thus $\hat\f$ as defined above from $\hat{\mathcal{F}}$  is the lift of an immersion 
into $\C P^2$ without complex points.

Hence the claims $(1)$, $(2)$ and $(3)$ follow from the last Theorem.
\end{proof}

\section{Ruh-Vilms type theorems}\label{sc:Ruh-Vilms}

In the famous Theorem of Ruh-Vilms \cite{RV}, for immersions into $\R^3$,  one proves that the Gauss map into $S^2$ of an immersion into $\R^3$ is harmonic if and only if the original immersion has constant mean curvature.
We will generalize this situation to minimal surfaces without complex points in $\C P^2$ and to minimal Lagrangian surfaces in $\C P^2$.

In our discussion of minimal surfaces in $\C P^2$ without complex points and of minimal Lagrangian surfaces in $\C P^2$ we restricted to liftable surfaces and thus moved the discussion primarily to surfaces defined on some contractible domain $\D \subset \C.$
Therefore, in the following sections we will exclusively consider immersions defined on $\D$.

\subsection{Various bundles}
We first introduce
three $6$-symmetric spaces of dimension $7$
which are bundles over $S^5$. Our approach applies and extends ideas of \cite{McIn} to our case.
 
 We consider the spaces $FL_1$, $FL_2$, and $FL_3$.
 We first choose a natural basis 
 \[
  e_1 = (1, 0, 0)^T,\quad  e_2 = (0, 1, 0)^T,\quad  e_3= (0, 0, 1)^T
 \]
 of $\C^3$. 

 $(1) \; FL_1:$ 
We now consider $\C^3$ as the real $6$-dimensional symplectic vector space 
 given by the symplectic form $\Omega = \Im \langle \;,\; \rangle$. 
 Then the family of (real) oriented 
 Lagrangian 
 subspaces of  $\C^3$  form a submanifold of the real Grassmannian $3$-spaces of $\C^3$, 
 that is, they form the  \emph{ Grassmannian 
 manifold $\LGr(3, \C^3)$ of oriented Lagrangian subspaces.}  It is easy to see that
 $\LGr(3, \C^3)$ can be represented as the homogeneous space 
 $\U / \SO$. 
 In this paper we use the special orthogonal matrix group $\SO$ as the 
 connected subgroup of $\SU$
 corresponding to the sub-Lie-algebra  of $  \mathfrak{su}_3 $  given by
\[
 \mathfrak{so}_3 = \left\{
 \begin{pmatrix}
 i a & 0 & y \\
 0 & -i a & \bar y \\
 - \bar y & -y  & 0
 \end{pmatrix}\;\Big|\;
 a \in \R, \, y \in \C
\right\}\subset  \mathfrak{su}_3,
\]
 which is isomorphic to the standard $\mathfrak{so}_3$ by 
 the automorphism $X \mapsto \ad (H)(X)$, where
\[
 H =  \begin{pmatrix}
 \frac{1-i}{2} &  \frac{1+i}{2} & 0 \\
 \frac{1+i}{2} &  \frac{1-i}{2} & 0 \\
0 & 0  & 1
 \end{pmatrix}.
  \]
  
The orbit of $\SU$ in $\LGr(3, \C^3)$ through the point $e \SO$ will be  called 
 \emph{special Lagrangian Grassmannian} and it will be denoted by 
 $\SLGr(3, \C^3)$.
 The elements in this orbit will be called \emph{oriented  special Lagrangian 
subspaces} of $\C^3$.

We summarize this by 
\begin{Proposition}\label{Prp:Grass}
  $\SU$ acts transitively on  $\SLGr(3, \C^3)$, 
  and we obtain
 \[
\SLGr(3, \C^3) = \SU / \SO.
 \]
The base point $e\SO$ corresponds to the real Lagrangian subspace of $\C^3$ given by $H^{-1} \R^3$.
\end{Proposition}
Next we define 
\begin{equation*}
FL_1 = \{ (v,V)\;|\; v \in S^5, \; v \in V, \;
 V  \in  \SLGr(3, \C^3)\}.
\end{equation*}
 It is easy to verify that $\SU$ acts (diagonally) on $FL_1$.
 Note that the natural projection from $FL_1$ to $\CP$ is a 
 Riemannian submersion which is equivariant under the natural group actions.
 Since $S^5 = \SU/\SUt$, where  $\SUt$ means 
 $\SUt \times \{1\}$, 
 the stabilizer at 
 \[
  (e_3, \operatorname{span}_{\R}\{ \tilde e_1, \tilde e_2,  e_3 \}) \in FL_1
 \quad \mbox{with}\quad
\tilde e_1 = \left(\frac{1+i}{2}, \frac{1-i}{2},0\right)^T, \quad 
\tilde e_2 = \left(\frac{1-i}{2},\frac{1+i}{2},0\right)^T
 \]
  is clearly given by
 $\SUt \cap \SO$, that is
 \[
  \Uone =\{(a, a^{-1}, 1)\;|\; a \in S^1 \}.
 \] 
 Therefore 
\[
 FL_1 = \SU/\Uone.
\]

 $(2)\; FL_2:$
 For the definition of $FL_2$, 
 we consider certain \emph{special regular complex flags} in $\C^3$.
 Here by a regular complex flag  $\mathcal{Q}$ we mean a sequence of four 
 complex subspaces, 
 $Q_0 = \{0\} \subset Q_1 \subset Q_2 \subset  Q_3 = \C^3$ 
 of $\C^3$, 
 where $Q_j$ has  complex dimension $j$.
 We then define the notion of a \emph{special regular complex flag} 
 in $\C^3$ over $q \in S^5$  
 by requiring that we have a  regular complex flag in $\C^3$,  
 where the space $Q_1$  satisfies $Q_1 = \C q$.
 Thus we define
 \begin{equation*}
FL_2 = 
 \left\{ (w,\mathcal{W})\mid  
\begin{array}{l}
\text{$w \in S^5$, $\mathcal W$ is 
 a special regular complex}  \\
\text{ flag over $w$ in $\C^3$ satisfying $W_1 = \C w$}
\end{array}
 \right\}.
\end{equation*}
  The definition of a special flag means that  for a given vector $q \neq 0$ in $\C^3$ one can find three pairwise orthogonal vectors $q_1, q_2, q_3 \in \C^3$ with
 $q_3 = \frac{q}{|q|}$ such that 
the vectors $q_1,q_2$ and $q_3$ represent the same orientation as $\tilde{e}_1, \tilde{e}_2, e_3$.
 By an argument similar to the previous case we conclude 
 that  $\SU$  acts transitively on the family of special flags.
 Moreover, the stabilizer of the action at the point 
 $(e_3, 0 \subset \C e_3 \subset  \C e_3 \oplus \C \tilde{e}_2
 \subset  \C e_3 \oplus \C \tilde{e}_2  \oplus \C \tilde{e}_1)$
 is again given by $\SO \cap \di$, where $\di$ denotes the set of all
 diagonal matrices in $\SU$. Thus it is again $\Uone$ and we have altogether shown
\begin{Proposition}\label{Prp:Flags}
 $\SU$ acts transitively on $FL_2$, and $FL_2$ can be 
 represented as
\[
  FL_2 = \SU / \Uone.
\]
 \end{Proposition}
 Note that the natural projection from $FL_2$ to $\CP$ 
 is a Riemannian submersion which is equivariant under the natural group actions.

 $(3)\; FL_3:$ Finally, using the isometry group $\SU$ of $S^5$,
 we can directly define a homogeneous space $FL_3$ as
\begin{equation}\label{eq:FL3}
FL_3 = \left\{ U P \;U^T\;\Big|\; \mbox{$U \in \SU$ and $P
 =\begin{pmatrix}
   0 & \epsilon^2 & 0 \\
   \epsilon^4 &0 & 0 \\
   0 & 0 & 1
  \end{pmatrix}$}\right\},
\end{equation}
 where $\epsilon = e^{\pi i/3}$.
\begin{Theorem}\label{Thm:3.3}
We  retain the assumptions and the notion above. 
 Then the following statements hold:
\begin{enumerate}
\item The spaces $FL_j$  $(j = 1,2,3)$ are homogeneous under the natural action of 
 $\SU$.
\item The homogeneous space $FL_j$ $(j = 1,2,3)$ can be represented as
 \[
       FL_j = \SU/ \Uone,
 \quad \mbox{where}\quad 
\Uone = \{ \di (a,a^{-1}, 1)\;|\; 
 a\in S^1 \}.
 \]
 In particular they are all $7$-dimensional.
\end{enumerate}
\end{Theorem}
\begin{proof}
The statements clearly follow from Proposition \ref{Prp:Grass}, 
 Proposition \ref{Prp:Flags} and  the definition of $FL_3$ in \eqref{eq:FL3},
where  the stabilizer at $P$ is easily computed as $\Uone$.
\end{proof}


\begin{Corollary}\label{cor:3.4}
 The homogeneous spaces $FL_j \;(j=1, 2, 3)$ are $6$-symmetric spaces.
Furthermore, they are naturally equivariantly diffeomorphic.
\end{Corollary}
\begin{proof}
 First we note that the group $G^{\R} = \SU$ has the
 complexification $G = \rm{SL}_3 \C$ and is the fixed point set group of the 
 real form involution $\tau$ given in \eqref{eq:tauonG0}.

 We show that $FL_3$ is a $6$-symmetric space. 
 First note that the stabilizer 
 \begin{equation}\label{eq:stabFL3}
 \textrm{Stab}_{P} = \{ X \in \SU \;|\; X P \; X^T = P\} 
 \end{equation}
 at the point $P$ of $FL_3$
 is $\Uone$.
We already know  that the order $6$-automorphism $\sigma$ 
 of $\SU$ given
 in \eqref{eq:sigmaonG} and the real form involution $\tau$
 commute.
 Moreover,  a direct computation shows that the fixed point set of 
 $\sigma$ in $\SU$ is $\Uone$. 
 Thus $\textrm{Stab}_{P}$ satisfies 
 the condition in Definition  \ref{def:k-symmetric}. 
Hence  $FL_3$ is $6$-symmetric space  in the sense of Definition 
 \ref{def:k-symmetric}. 
 Furthermore, since all the spaces $FL_j$ are $\SU$-orbits with the same stabilizer, the identity homomorphism of $\SU$ descends for any pair of homogeneous spaces 
 $FL_j$ and $FL_m$ 
 to a diffeomorphism 
 \[
  \phi_{jm}: FL_m \rightarrow FL_j
 \] 
 such that for any $g \in \SU$ and $p\in FL_m$  we have 
\[
 \phi_{jm} (g.p) = g.\phi_{jm} (p).
\]
 As a consequence, also $FL_1$ and $FL_2$ are  $6$-symmetric spaces.
\end{proof}
\subsection{Projections from the bundles}
 We have seen that the homogeneous spaces $FL_j$ $(j=1, 2, 3)$ are 
 $7$ dimensional $6$-symmetric spaces. In this section we define
 natural projections from $FL_j$ to several homogeneous spaces.

 First from $FL_1$, we have a projection to $\SLGr(3, \C^3)$ given by
\[
 FL_1 \ni (v, V) \longmapsto V \in  \SLGr(3, \C^3).
\]
 It is easy to see that $\SLGr(3, \C^3)$ is a symmetric space with 
 the involution $\sigma^3$ defined in \eqref{eq:sigma3}.

 Next from $FL_2$, we have a projection to a full flag manifold:
\[
  FL_2 \ni (w, W) \longmapsto W \in  Fl_2,
\]
 where $Fl_2$ is defined as
\[
 Fl_2 = \{\mathcal W \;|\; \mbox{$\mathcal W$ is a regular complex flag in $\C^3$} \}.
\]
 It is easy to see that $Fl_2$ is a $3$-symmetric space with 
 the involution $\sigma^2$ stated  in \eqref{eq:sigma2}.
 
 Finally from $FL_3$, we have two projections. We first 
 let $k \in \textrm{Stab}_{P}$ as in \eqref{eq:stabFL3} with 
\[
P=
  \begin{pmatrix}
 0  & \epsilon^2 & 0 \\
 \epsilon^4 & 0 & 0 \\
0  & 0  & 1
 \end{pmatrix}, \quad \epsilon = e^{\pi i /3},
\]
  then a
 straightforward computation shows that 
\[
 k P P^T k^{-1} =   k P k^T P^T    = PP^T, \quad
 k P P^T P  k^{T} =  P (k^T)^{-1} P^T k^{-1} P  =  PP^T P.
\]
 Therefore we have two projections 
\begin{eqnarray*}
 &FL_3 \ni U P U^T \longmapsto U (P P^T) U^{-1} \in \widetilde {Fl_2}, \\
 &FL_3 \ni U P U^T \longmapsto U (P P^T P) U^T \in \widetilde \SLGr(3, \C), 
\end{eqnarray*}
 where $\widetilde {Fl_2}$ and $\widetilde \SLGr(3, \C)$ are defined as
 \[
\widetilde {Fl_2} = \{U (P P^T) U^{-1}\;|\; U \in \SU\}, \quad
\widetilde \SLGr(3, \C) = \{U (P P^T P) U^{T}\;|\; U \in \SU\}.
 \]
 Note that it is easy to compute 
 \[
  P P^T = 
\begin{pmatrix}
 \epsilon^4 & 0 & 0 \\
 0 & \epsilon^2 & 0 \\
 0 & 0 & 1
\end{pmatrix}, \quad 
  P P^T P = 
\begin{pmatrix}
 0 & 1 & 0 \\
 1 & 0 & 0 \\
 0 & 0 & 1
\end{pmatrix}, 
 \]
 and the stabilizer in $\SU$ at $P P^T$ of $\widetilde{Fl_2}$ and 
 the stabilizer in $\SU$ at $P P^T P $ of $\widetilde{\SLGr}(3, \C)$ are 
\[
\textrm{Stab}_{PP^T} = \textrm{D}_3, \quad
 \textrm{Stab}_{PP^TP} = \SO,
\]
 where 
\[
 \textrm{D}_3 = 
\{ \di (a_1, a_2, a_3) \in \SU \},
\]
and where $\textrm{Stab}_{PP^TP}$ is exactly the same group as the stabilizer
of  $\SLGr(3, \C)$. Thus  $\SLGr(3, \C)$ and $\widetilde{\SLGr}(3, \C)$ are naturally 
equivariantly diffeomorphic. An analogous argument applies to  $Fl_2$ and  $\widetilde{Fl_2}$. Now  the stabilizer of $\widetilde{Fl_2}$ is determined by  the matrix characterizing $\sigma^2$, whence $\widetilde{Fl_2}$ (and thus $Fl_2$) is the 3-symmetric space associated with $\sigma^2$. Similarly, $\SLGr(3, \C)$ (and thus $\widetilde{\SLGr}(3, \C))$ 
is the symmetric space associated with $\sigma^{3}$.

Thus we obtain:

\begin{Theorem}\label{Thm:3.5}
 We  retain the assumptions and the notion above. 
 Then the following statements hold$:$
\begin{enumerate}
\item The spaces $\SLGr(3, \C)$ and $\widetilde{\SLGr}(3, \C)$ are 
naturally equivariantly diffeomorphic 
 symmetric spaces relative to $\sigma^3$, and they are $5$-dimensional.
\item The spaces $Fl_2$ and $\widetilde{Fl_2}$ are 
naturally equivariantly diffeomorphic 
 $3$-symmetric spaces relative to $\sigma^2$, and they are $6$-dimensional.

\item The homogeneous spaces
 $\SLGr(3, \C)$ and $\widetilde{\SLGr}(3, \C)$, 
 and $Fl_2$ and $\widetilde{Fl_2}$ can be represented as
 \begin{align*}
\SLGr(3, \C) &= \SU/\SO,\quad \widetilde{\SLGr}(3, \C)  = \SU/\SO, \\ 
 Fl_2 &= \SU/D_3, \quad  \widetilde{Fl_2} = \SU/D_3.  
\end{align*}
\end{enumerate}
\end{Theorem}
 We now define several projections:  
\begin{gather*}
\pi_j: \SU \to FL_j, \quad (j =1, 2, 3), 
\intertext{and}
\tilde \pi_1: FL_1 \to \SLGr(3, \C), \quad 
\tilde \pi_2: FL_2 \to Fl_2, \quad 
\tilde \pi_{3, 1}: FL_3 \to \widetilde{\SLGr}(3, \C), \quad 
\tilde \pi_{3, 2}: FL_3 \to \widetilde{Fl}_2.
\end{gather*}
 Schematically, we have the  following diagram:
\[
  \begin{diagram}
    \node[3]{\SU} \arrow{sw,l}{\pi_1} \arrow{s,l}{\pi_3} \arrow{se,l}{\pi_2} \\
    \node[2]{FL_1}\arrow{sw,l}{\tilde \pi_1} \node{FL_3}\arrow{sw,l}{\tilde\pi_{3,1}} \arrow{se,l}{\tilde\pi_{3,2}}
  \node{FL_2}\arrow{se,l}{\tilde \pi_2}  \\
    \node{\SLGr(3, \C)\;\;\cong}   \node{\widetilde{\SLGr}(3, \C)} 
  \node[2]{\widetilde{Fl_2}\quad \quad \cong} \node{Fl_2}\\
  \end{diagram}
\]

\subsection{Gauss maps}
 We now define three Gauss maps for any liftable immersion  $f : M \rightarrow \CP$
 without complex points, with $M$ a Riemann surface. 
 For our purposes in this subsection it will suffice to consider the  lift $\f$ to a map $\tilde{f} : \D \rightarrow \CP$, where $\D$ denotes the universal cover of $M$.
 Therefore we will assume from now on $M = \D$, unless the opposite is stated explicitly.
 
 So let us thus assume that $f$ is defined on a simply connected domain $\D \subset \C$ and that
 $\f$ is a special lift of $f$. Then we define the frame 
 $\mathcal F: \D \to  \U$ as in Theorem \ref{det=1} such that 
 $\det \mathcal F =1$, that is,
\begin{equation}\label{eq:normalizedframe}
 \mathcal F: \D \to \SU.
 \end{equation}
 $\mathcal F$ will be called the \textit{normalized frame}.
 Note that the function $\rho$ has been chosen now and may not coincide with $\rho_0$ as in Proposition \ref{rho=rho0}.

\begin{Definition}\label{dfn:normalizedGauss}
\mbox{}
Retain the above notation.
\begin{enumerate}
 \item 
 Consider the projections  
 $\pi_j \circ \mathcal F : \D \to FL_j$ $(j =1, 2, 3)$, where $\pi_j: \SU \to FL_j$.
 Then 
 \[
  \g_j = \pi_j \circ \mathcal  F \quad (j =1, 2, 3)
 \]
 will be called the \emph{Gauss map} of $f$ with values in $FL_j$. 
 These Gauss maps are clearly well-defined on $\D$ (independent of the choice of coordinates).
 \item Furthermore we follow the Gauss maps with the projections from $FL_j$ to 
 $\SLGr(3, \C), Fl_2,$ 
 $\widetilde{\SLGr}(3, \C)$ or $\widetilde{Fl_2}$ respectively as discussed just above, i.e.,
\[
{\mathcal H_i} = \tilde \pi_{i} \circ \pi_{i} \circ \mathcal  F \quad ({i} =1, 2), 
 \quad   
 {\mathcal{H}}_{3, i} = \tilde \pi_{3, i} \circ \pi_3 \circ \mathcal  F \quad (i =1, 2).
\]
 These maps will be called the \emph{Gauss maps} of $f$ with values in 
 $\SLGr(3, \C)$, $Fl_2$, $\widetilde{\SLGr}(3, \C)$ or $\widetilde {Fl_2}$ respectively.
\end{enumerate}

\end{Definition}
 Our definitions were a priori not very geometric.
 But by following \cite{McIn} we find analogously
 $7$ obvious geometric
 interpretations of the Gauss map.

 {\bf For $FL_1$ and $\SLGr(3, \C)$: } Let $\mathcal{G}_1 : \D \to FL_1$ be given by 
\[
 p \mapsto (\f(p), \>
 \operatorname{span}_{\R} \{ \tilde \xi(p), 
 \tilde \eta(p), \f(p) \}),
\]
 where 
\[
 \tilde \xi = -i e^{- \omega/2} \sqrt{a}^{-1}\xi, \quad 
 \tilde \eta = -i e^{- \omega/2} \sqrt{b}^{-1}\eta,
\]
 and $\f$ is a lift of $f$ such that $\det \mathcal F = 1$.
Furthermore, the Gauss map $\mathcal H_1 : \D \to \SLGr(3, \C)$ 
 is given by $\tilde{\pi}_1 \circ \mathcal G_1$, i.e.,
\[
 p \mapsto (\operatorname{span}_{\R} \{ \tilde \xi(p), 
 \tilde \eta(p), \f(p) \}).
\]

 {\bf For $FL_2$ and $Fl_2$:} Let $\mathcal{G}_2: \D \to FL_2$ be given  by 
\[
p \mapsto  
 (\f(p), 0 \subset \C \f(p) \subset  \C \f(p) \oplus \C \tilde \xi(p) 
 \subset  \C \f(p) \oplus \C \tilde \xi(p)  \oplus \C \tilde \eta (p)).
\]
Furthermore, the Gauss map $\mathcal H_2 : \D \to Fl_2$ 
 is given by $\tilde{\pi}_2 \circ \mathcal G_2$, i.e.,
\[
 p \mapsto (0 \subset \C \f(p) \subset  \C \f(p) \oplus \C \tilde \xi(p) 
 \subset  \C \f(p) \oplus \C \tilde \xi(p)  \oplus \C \tilde \eta (p)).
\]
 
 {\bf For $FL_3$, $\widetilde{\SLGr}(3, \C)$ and $\widetilde{Fl_2}$:} We observe that one can represent the Gauss map ${\g}_3$ by using 
 the frame $\mathcal  F$ defined in 
 Theorem \ref{det=1} as 
\[
 \mathcal{G}_3 = \mathcal  F  P\>\mathcal F^T, \quad \mbox{with} \quad P=
\begin{pmatrix}
 0 & \epsilon^2 & 0 \\
 \epsilon^4 & 0 & 0\\
 0 & 0 & 1
\end{pmatrix},
\]
 where $\epsilon = e^{\pi i/3}$.
Furthermore, the Gauss maps $\mathcal{H}_{3, 1} : 
 \D \to \widetilde{\SLGr}(3, \C)$ 
and $\mathcal H_{3, 2} : \D \to \widetilde{Fl_2}$
are given by $\tilde\pi_{3, i} \circ \mathcal G_3$, i.e.,
\[
{\mathcal{H}_{3, 1}:} \, p \mapsto \mathcal F (P P^T P) \mathcal F^{T}, 
\quad {\mathcal{H}_{3, 2}:}  \, p \mapsto \mathcal F (P P^T) \mathcal F^{-1}. 
\]
 
\subsection{Ruh-Vilms type theorems  associated with the Gauss maps}

We finally arrive at Ruh-Vilms type theorems.
\begin{Theorem}[Ruh-Vilms theorems for $\sigma, \sigma^2$ and $\sigma^3$]\label{Thm:3.6}
 With the notation used above we consider 
for any liftable immersion into $\C P^2$ the Gauss maps$:$
\begin{enumerate}
\item $\mathcal{G}_{j} : M \rightarrow FL_{j}$ for $j=1, 2, 3$,

\item $\mathcal H_2 = \tilde \pi_2 \circ \mathcal{G}_2: 
 M \rightarrow  Fl_2$ and 
 $\mathcal H_{3, 2} = \tilde \pi_{3, 2} \circ  \mathcal{G}_3 : 
 M \rightarrow  \widetilde{Fl_2}$,

\item $\mathcal H_1 = \tilde \pi_1 \circ  \mathcal{G}_1 : 
 M \rightarrow  \SLGr(3, \C)$ and 
 $\mathcal H_{3, 1} = \tilde \pi_{3, 1} \circ  \mathcal{G}_3 : 
 M \rightarrow  \widetilde{\SLGr(3, \C)}$.

\end{enumerate}
Then the following statements hold$:$
\begin{enumerate}
 \item  $\mathcal{G}_j$ $(j=1, 2, 3)$ is primitive harmonic map into $FL_{j}$  
 if and only if  $\mathcal F$ is primitive harmonic relative to $\sigma$ 
 if and only if the corresponding surface is a minimal Lagrangian 
 immersion into $\C P^2$.

\item $\mathcal H_2$ or $\mathcal H_{3, 2}$ is primitive harmonic  in $Fl_2$ or $ \widetilde{Fl_2}$
 if and only if $\mathcal F$  
is primitive harmonic relative to $\sigma^2$ if and only if 
the corresponding surface is a minimal immersion into
$\C P^2$ without complex points.

\item  $\mathcal H_1$ or $\mathcal H_{3, 1}$ is primitive harmonic map
 into $\SLGr(3, \C)$ or $\widetilde{\SLGr(3, \C)}$ 
 if and only if  $\mathcal F$ is primitive harmonic
 relative to $\sigma^3$ if and only if 
 the corresponding surface is either a minimal Lagrangian immersion 
 or a flat homogeneous immersion  into $\C P^2$.
\end{enumerate}

\end{Theorem}

\begin{proof}
The first equivalence in $(1)$ is due to the definition of primitive harmonicity 
 into a $k$-symmetric space. The second equivalence has been stated in 
 Theorem \ref{equivprimitive}. The proofs for $(2)$ and $(3)$ are similar. 
\end{proof}
\begin{Remark}
 We would like to point out that the result above is not contained in \cite{McIn}.
\end{Remark}

\appendix\section{}\label{appendix}

In this appendix, we discuss the liftability of an immersion $f: M \to \C P^2$
 into $S^5$.
\subsection{The non-compact case}

 \begin{Theorem}
 Let $\D \subset \C$ be a simply-connected domain and $f:\D \rightarrow \C P^2$ an immersion without complex points.
 Let $\f_0: \D \rightarrow S^5$ be a lift of $f$ and $\mathcal{F}(\f_0 )$ the corresponding frame. 
 Then
 \begin{enumerate}
  \item[a)] There exists some smooth function $\delta: \D \rightarrow S^1$ such that 
 $\det  \mathcal{F}(\delta\f_0 ) = 1$.
 
 \item[b)] Any two lifts $\f_0$ and  $\f_1$ of $f$ for which   $\det  \mathcal{F}(\f_0 ) = 1$ and 
  $\det  \mathcal{F}(\f_1 ) = 1$ differ by  a cubic root of unity.
 \end{enumerate}

  \end{Theorem}
 
 \begin{proof}
 a) Put $\delta_0 =  \det \mathcal{F}(\f_0 ).$ Then $\delta_0 : \D \rightarrow S^1$ is smooth. Since $\D$ is simply-connected we can define the smooth function
 $\delta = \delta_0^{-1/3}: \D \rightarrow S^1$, then  $\det  \mathcal{F}(\delta \f_0 ) = 1$.
 
 b) Assume  $\det  \mathcal{F}(\f_0 ) = \det  \mathcal{F}(\f_1 ) = 1.$ Since $\f_0$ and $\f_1$ are both lifts of $f$ on $\D$, there exists some smooth function $h: \D \rightarrow S^1$ such that 
 $\f_1 = h \f_0$ holds. Then  $ \det  \mathcal{F}(\f_1 ) = \det  \mathcal{F}(h \f_0 ) = 1$
 implies $h^3 = 1$. Hence $h$ is a constant.
 \end{proof}
 
 From this we derive
 
 \begin{Theorem}
 Let $M$ be a non-compact Riemann surface and $f:M \rightarrow \C P^2$ an immersion without complex points. Then there exists a global lift $\f : M \rightarrow S^5$.
 \end{Theorem}
 
 \begin{proof}
 Let $\{ U_\alpha \}$ be an open covering of $M$ by open contractible subsets (disks).
 Then on each $U_\alpha$ there exists some lift $\f_\alpha: U_\alpha \rightarrow S^5$ of 
 $f_{| U_\alpha}$ such that  $ \det  \mathcal{F}( \f_\alpha ) = 1$ holds.
 On the intersection $U_\alpha \cap U_\beta$ we consider a connected 
 component $C_{\alpha \beta}^\iota.$ Then $\f_\alpha = h_{\alpha \beta}^\iota 
 \f_\beta$ 
 on $ C_{\alpha \beta}^\iota $ with some unique smooth function 
 $h_{\alpha \beta}^\iota : C_{\alpha \beta}^\iota  \rightarrow  S^1$.
 Now $\mathcal{F}(\f_{\alpha}) =  \mathcal{F}(  h_{\alpha \beta}^\iota  
 \f_{\beta}) =  (h_{\alpha \beta}^\iota)^{3} \mathcal{F}(\f_{\beta} ) $ and the requirement that  $\det  \mathcal{F}(\f_\alpha ) = \det  \mathcal{F}(\f_\beta ) = 1$ holds implies that $ h_{\alpha \beta}^\iota $ is a cubic root of unity.
 In particular,  $h_{\alpha \beta}^\iota$ is constant and thus holomorphic.
 Altogether we obtain $f_\alpha = h_{\alpha \beta} f_\beta$ on $U_\alpha \cap U_\beta$ with a holomorphic function   $h_{\alpha \beta}$ on $U_\alpha \cap U_\beta.$
 It is easy to verify that the family of   $h_{\alpha \beta}$  is a cocycle. 
 Since we have assumed that $M$ is non-compact, the cocycle 
 $\{ h_{\alpha \beta} \}$ splits (see, e.g. \cite{Forster}, Corollary 30.5). Therefore there exist holomorphic functions 
 $w_\alpha$ on $U_\alpha$  satisfying $h_{\alpha \beta} = w_\alpha^{-1} w_\beta.$
 As a consequence the family  of $w_\alpha \f_\alpha$ defines a globally defined function 
 $\f : M \rightarrow S^5$ and thus a global lift of $f$.
 \end{proof}
 
 \begin{Remark}
\mbox{}
\begin{enumerate}
 \item The frame corresponding to $\f$, as in the last theorem,  generally speaking only makes sense if $\f$ is defined on a simply-connected open subset of $\C.$ As a consequence, the condition $\det  \mathcal{F}(\f) = 1$ only makes sense on $\D$.
 
\item  If $M$ is compact, then one can repeat the argument above with a meromorphic splitting.
 Hence one needs to admit (finitely many) singularities in the global lift $\f$.
\end{enumerate}

 \end{Remark}
 
 
 \subsection{The general case}
Recall that we assume that $M$ is different from $S^2$. We use this right below,
  when we state that  $\tilde{f}: \D \rightarrow \C P^2$  has a lift 
  $\tilde{\f}: \D \rightarrow S^5.$
   This is proven by considering the pull back bundle and using that $\D$ is contractible.
 \begin{Proposition}
 Let  $f : M \rightarrow \C P^2$ be an immersion without complex points and $\tilde{f} : \D \rightarrow \C P^2$ denote the lift $\tilde{f} = f \circ \tilde \pi $ of $f$ to the universal cover $\tilde \pi : \D \rightarrow M$.
 Then  $\tilde{f} $ has a lift  $\tilde{\f} : \D \rightarrow S^5 $ and the following statements hold
 \begin{enumerate}
 \item For $\gamma \in \pi_1(M)$, acting on $\D$ by M\"{o}bius transformations, we obtain that also  $\gamma^*\tilde{\f} $  is a lift of  $\tilde{f} .$
 
 \item For all $\gamma \in \pi_1 (M)$ we have 
 $(\gamma^*\tilde{\f} ) (z, \bar{z}) = c(\gamma, z, \bar{z}) \tilde{\f} (z, \bar z)$  with $c$ taking values in $S^1$.
 
 \item After  multiplying $\tilde{\f}$ by a scalar multiple in $S^1$ we can assume without loss of generality that 
  $\mathcal{F}(\tilde{\f} ) $ is contained in $\SU$.
  
  \item For $\tilde{f}$ as just above and $\gamma \in \pi_1 (M)$ we obtain
  \begin{equation}
  \gamma^*( \mathcal{F}(\tilde{\f} ) )(z, \bar z) = c(\gamma, z ,\bar z) \mathcal{F}(\tilde{\f} ) (z, \bar z) k(\gamma, z ,\bar z),
  \end{equation}
  with $k(\gamma, z ,\bar z) = \di (|\gamma'| / \gamma', |\gamma'| / \bar{\gamma}',1)$, where $\gamma^{\prime}=\gamma_z$.
 \end{enumerate}
 
  \end{Proposition} 
  
  \begin{proof}
$(1)$ This can be deduced directly after composing  these maps with the Hopf fibration.
  
$(2) $ This just rephrases that both maps are lifts of $\tilde f$.
  
 $(3)$ As pointed out in the remark above  this can be done since the frame is defined on a simply-connected domain.
  
$(4)$ This claim will follow from a series of simple statements:
  
First by the chain rule we have 
$(\gamma^* \tilde{\f})_z  = \partial_z( \tilde{\f} \circ \gamma) =  \tilde{\f}_z \circ \gamma \cdot \gamma'$. Then it follows that
\begin{align*}
\gamma^* (\xi(\tilde{\f}) )&=\gamma^*\tilde{\f}_z-(\gamma^*\tilde{\f}_z\cdot \overline{\gamma^*\tilde{\f}}) \gamma^*\tilde{\f}\\
&=\frac{1}{\gamma^{\prime}}(\gamma^*\tilde{\f})_z-(\frac{1}{\gamma^{\prime}}(\gamma^*\tilde{\f})_z
\cdot \overline{\gamma^*\tilde{\f}}) \gamma^*\tilde{\f}\\
&=\frac{1}{\gamma^{\prime}} \{ (c\tilde{\f})_z-((c\tilde{\f})_z \cdot \overline{c\tilde{\f}}) c \tilde{\f}\}\\
&=\frac{1}{\gamma^{\prime}} \{ c_z\tilde{\f} +c\tilde{\f}_z -((c_z \tilde{\f}+c\tilde{\f}_z)\cdot \bar{\tilde{\f}}) \tilde{\f} \}\\
&=\frac{1}{\gamma^{\prime}} c \xi(\tilde{\f}).
\end{align*}
  That is, $\gamma^*( \xi(\tilde{\f}) ) = (\gamma')^{-1} c(\gamma, \cdot) \xi(\tilde{\f})$.
  Similarly, we obtain
  $\gamma^*( \eta(\tilde{\f}) ) = (\bar {\gamma}')^{-1} c(\gamma, \cdot) \eta (\tilde{\f})$.
  On the other hand, since $\gamma$ acts on $\D$ by isometries, 
   $e^{\omega} dz d \bar{z} = \gamma^* ( e^{\omega} dz d \bar{z} ) = \gamma^*(e^\omega) |\gamma '|^2 dz d \bar{z}$.
   Moreover, the functions $a$ and $b$ are independent of the choice of $\tilde{\f}$.
  Putting this together we obtain for the frame  $\mathcal{F}(\tilde{\f} ) $ the claim.
  \end{proof}

 \begin{Corollary}
 In view of the fact that we can assume $ \det  \mathcal{F}( \tilde{\f}) = 1,$ the transformation formula above for the frame implies $c(\gamma, z , \bar z)^3 = 1$ and thus 
 \begin{equation}
 c(\gamma, z , \bar z) = c(\gamma) \in S^1
 \end{equation}
 for all $\gamma \in  \pi_1(M)$. 
 In particular, $c: \pi_1 (M) \rightarrow S^1$ is a homomorphism with values in the  group 
 $\mathbb{A}_3$ of cubic  
 roots of unity, whence the image of $c$ is either 
 $\{e \}$ or all of $\mathbb{A}_3$.
  \end{Corollary}
  
  From this we derive the following
  \begin{Theorem}
  Let $M$ be a Riemann surface, different from $S^2$, and $f:M \rightarrow \C P^2$ an immersion without complex points. 
  Let $\tilde \pi: \D \rightarrow M$ denote the universal covering of $M$ and 
  $\tilde{f} = f \circ \tilde \pi : \D \rightarrow \C P^2$ the natural lift of $f$ to $\D$.
  Let $\tilde{\f} : \D \rightarrow S^5 $ denote a lift of $\tilde{f}$  satisfying 
  $ \det  \mathcal{F}( \tilde{\f}) = 1.$ 
  Let $c: \pi_1 (M) \rightarrow S^1$ denote the homomorphism induced by $\tilde{\f}$ and put
  $\Gamma = \ker (c)$. Furthermore, define the Riemann surface   
  $\hat{M} = \Gamma \backslash  \D$.
  Then the following statements hold$:$
  
\begin{enumerate}
 \item[a)] The definitions above induce naturally a sequence of coverings 
  \begin{equation}
 \xymatrix{\D \ar[r]^{\hat{\pi}}&\hat{M}\ar[r]^{\tau}&M,}
  \end{equation}
   where the first map is denoted by $\hat{\pi}$ and the second map is denoted by $\tau$. Recall that our definitions imply $\pi = \tau \circ \hat{\pi}$.
  Moreover, the covering map $\tau$ has either order $1$ or order $3$.
  
\item[b)] Putting $\hat{f} = f \circ \tau: \hat{M} \rightarrow  \C P^2$ we obtain the commuting diagram,
  \[
  \begin{diagram}
    \node{\D}  \arrow{r,l}{\tilde{\f}} \arrow{s,l}{\hat{\pi}} \node{S^5} \arrow{s,r}{\pi}  \\
    \node{\hat{M}}\arrow{s,l}{\tau}\arrow{ne,l}{\hat{\f}} \arrow{r,l}{\hat{f}}\node{\C P^2} \\
    \node{M} \arrow{ne,r}{f}
  \end{diagram}
\]
  where  $\hat{\f} : \hat{M} \rightarrow S^5$ is the naturally global lift of $\hat{f}$.
  Then, either $\hat{M} = M$ and $f$ itself has a global lift or $\tau: \hat{M} \rightarrow M$ has order three and $\hat{M}$ has the global lift $\hat{\f} $. 
\end{enumerate}

  \end{Theorem}
  
  \begin{proof}
  Since the image of $c$ is either only the identity element of $S^1$ or the full group of cubic roots, the kernel of $c$ either is all of $\pi_1 (M)$ or a subgroup $\Gamma$ satisfying $\mathbb{A}_3 \cong \pi_1(M) / \Gamma$.
  
  In the first case  $\hat{M} = M$ and $\hat{\f} $ actually is a global lift of $f$. In the second case, the map $\hat{f} : \hat{M} \rightarrow \C P^2$ has a global lift, namely 
  $\hat{\f} : \hat{M} \rightarrow S^5$.
 \end{proof}
 
 \begin{Corollary}
 Let $M$ be a Riemann surface different from $S^2$ and $f: M \rightarrow \C P^2$ an immersion without complex points.
 Then either $f$ has a global lift $\f : M \rightarrow S^5,$  or there exists a $3$-fold covering 
 $\tau: \hat{M} \rightarrow M$ of $M$ such that the immersion 
 $\hat{f} = f \circ \tau: \hat{M} \rightarrow \C P^2$ has a global lift, while the given 
 $f : M \rightarrow \C P^2$ has not.
 \end{Corollary}


\bibliographystyle{plain}
\def\cprime{$'$}

\end{document}